\documentclass[preprint,1p]{elsarticle}

\usepackage{epsfig,makeidx,color}
\usepackage{graphicx}
\usepackage{amsbsy}
\usepackage{amsmath,amsthm}
\usepackage{amssymb}
\usepackage{euscript,verbatim}
\usepackage{array,multirow}
\usepackage{hyperref}
\usepackage[numbers]{natbib}

\theoremstyle{plain}
\newtheorem{lemma}{Lemma}
\newtheorem{prop}{Proposition}
\newtheorem{theorem}{Theorem}
\newtheorem{corollary}{Corollary}
\theoremstyle{definition}

\newtheorem{definition}{Definition}
\newtheorem{example}{Example}

\newcommand{\Z}{\mathbb{Z}}
\newcommand{\R}{\mathbb{R}}
\newcommand{\Q}{\mathbb{Q}}

\DeclareMathOperator{\vol}{vol}

\DeclareMathOperator{\Gal}{Gal}

\begin{document}
\begin{frontmatter}

\title{Well-Rounded Twists of Ideal Lattices from Real Quadratic Fields}

\author[td]{Mohamed Taoufiq Damir\fnref{funding}}
\address[td]{Department of Mathematics and Systems Analysis, Aalto University, P.O.\ Box 11100, FI-00076 Aalto, Finland}
\ead{mohamed.damir@aalto.fi}
\fntext[funding]{Mohamed Taoufiq Damir is supported in part by the Academy of Finland, through grants \#276031, \#282938, and \#303819 awarded to C.\ Hollanti. }
\author[dk]{David Karpuk\corref{ca}}
\address[dk]{Departamento de Matem\'aticas, Universidad de los Andes, Carrera 1 \#18A-10, Edificio H, Bogot\'a, Colombia}
\ead{da.karpuk@uniandes.edu.co}
\cortext[ca]{Corresponding author}

\begin{abstract}
We study ideal lattices in $\R^2$ coming from real quadratic fields, and give an explicit method for computing all well-rounded twists of any such ideal lattice.  We apply this to ideal lattices coming from Markoff numbers to construct infinite families of non-equivalent planar lattices with good sphere-packing radius and good minimum product distance.  We also provide a complete classification of all real quadratic fields such that the orthogonal lattice is the only well-rounded twist of the lattice corresponding to the ring of integers.  

\end{abstract}


\end{frontmatter}

\section{Introduction and Background}

\subsection{Introduction}

Lattices play a central role in many areas of mathematics, with deep and extensive connections to number theory, algebraic coding theory, finite group theory, and Lie group theory.  Codebooks constructed from lattices are used in communication over channels with additive noise \cite[Chapter III]{conway_sloane}, and more recently, lattices arising from totally real number fields have found applications in communication over wireless channels \cite{oggier_belfiore}.

The utility of a given lattice for a given application is measured using some relevant invariant, such as the sphere-packing radius, the normalized second moment, etc.  In communications, large sphere-packing radius $\rho(\Lambda)$ is desirable to protect against additive Gaussian noise, while large \emph{minimum product distance} $N(\Lambda)$ (defined in the following subsection) is desirable to protect against fading, an effect of certain wireless channels analogous to an erasure in traditional coding theory.  Constructing lattice codebooks resistant to both noise and fading requires lattices for which $\rho(\Lambda)$ and $N(\Lambda)$ are large.

\emph{Ideal lattices}, those arising from the canonical embedding of an ideal in the ring of integers of a totally real number field, are usually the prime candidates for constructing lattices with good minimum product distance.  In \cite{oggier_belfiore}, the authors construct orthogonal lattices with good minimum product distance using ideal lattices, and in \cite{good_lattices,eva_suarez}, the authors realize several famous lattices, such as $D_4$, $E_8$, and the Leech lattice $\Lambda_{24}$ as ideal lattices. These constructions also involve \emph{twists} of ideal lattices, where a twist of a lattice $\Lambda$ is a lattice $T\Lambda$ where $T$ is a diagonal matrix with positive entries and determinant one.  Twisting a lattice leaves $N(\Lambda)$ unchanged but gives one the opportunity to increase $\rho(\Lambda)$.

While the constructions mentioned in the previous paragraph are indeed useful, they leave several interesting questions unanswered.  For example, given a dimension $n$, can we explicitly construct an infinite family of non-equivalent lattices $\Lambda$ such that $\rho(\Lambda)$ and $N(\Lambda)$ are uniformly bounded below by positive constants, i.e., are both large?  How do we explicitly compute all twists of an ideal lattice which have large sphere-packing radius?  We will answer these and related questions for ideal lattices coming from real quadratic fields.

To paint our approach in broad strokes we describe the problem geometrically following McMullen \cite{mcmullen}; see also the course notes \cite{mcmullen_notes}.  Let $\mathcal{S}_n$ be the space of all lattices $\Lambda\subset\R^n$ up to similarity.  Inside of $\mathcal{S}_n$ we have the \emph{well-rounded locus} $\mathcal{W}_n$ consisting of all well-rounded lattices.  For example, $\mathcal{S}_2$ can be identified with the fundamental domain of $SL_2(\Z)$ acting on the upper-half plane, and $\mathcal{W}_2$ is the ``bottom arc'' of $\mathcal{S}_2$.  For our purposes, well-rounded lattices will serve as an accessible class of lattices with decent sphere-packings; indeed, a classical theorem of Voronoi \cite{Vor1,Vor2} asserts that every local maximum of the sphere-packing density is a perfect and eutactic lattice, and all perfect lattices are well-rounded.   

Let $\mathcal{A}_n$ be the \emph{diagonal group} of all diagonal matrices with positive entries and determinant one and consider the orbit $\gamma(\Lambda) = \mathcal{A}_n\cdot \Lambda$ as a submanifold of $\mathcal{S}_n$.  The intersection $w(\Lambda) = \gamma(\Lambda)\cap\mathcal{W}_n$ is the set of all well-rounded twists of $\Lambda$, and is our main object of study.  Results of McMullen \cite{mcmullen} assure us that this is not a fool's errand: if $\Lambda$ is an ideal lattice from a totally real number field, then the orbit $\gamma(\Lambda)$ is compact, and for any lattice such that $\gamma(\Lambda)$ is compact the intersection $w(\Lambda)$ is non-empty.  It is worth mentioning a recent result of Levin et al \cite{closed_orbits}, which generalizes this statement to show that $w(\Lambda)$ is non-empty for any $\Lambda$ such that $\gamma(\Lambda)$ is closed.  Thus well-rounded twists of ideal lattices exist, but as far as the authors are aware, there is no general method for computing them even for $n = 2$. 

In \cite{lenny1,lenny2}, the authors make a detailed study of which ideal lattices are well-rounded, focusing especially on the case of lattices in $\R^2$.  Among other things, they show that a positive proportion of real quadratic fields $K$ contain an ideal in $\mathcal{O}_K$ whose ideal lattice is well-rounded.  However, a number of interesting ideal lattices are not well-rounded, for example the ideal coming from the ring of integers in $\Q(\sqrt{5})$.  To make them well-rounded, one considers twists.

The current paper has two main goals.  First, given an ideal lattice $\Lambda\subset\R^2$ coming from a real quadratic field, we wish to compute the set $w(\Lambda)$ of well-rounded twists explicitly.  Second, we wish to apply this explicit computation to infinite families of ideal lattices with large $N(\Lambda)$, to construct infinite families of well-rounded twists of ideal lattices which have large $\rho(\Lambda)$ and $N(\Lambda)$.  These infinite families will come from Markoff numbers.  We also arrive at a number of interesting secondary results, such as a complete classification of all rings of integers $\mathcal{O}_K$ in real quadratic fields such that the only well-rounded twist of the corresponding lattice is orthogonal, an upper bound on the number of well-rounded twists of an ideal lattice, and a proof that the ideal lattice in $\R^2$ with maximum $N(\Lambda)$ is the one coming from the ring of integers of $\Q(\sqrt{5})$.

\subsection{Lattices}

A \emph{lattice} $\Lambda$ is a discrete subgroup of $\R^n$ of rank $n$.  Equivalently, $\Lambda$ is the $\Z$-span of a set $B = \{v_1,\ldots,v_n\}$ of linearly independent vectors in $\R^n$, called a \emph{basis} of $\Lambda$.  One can often define $\Lambda$ as $\Lambda = M_\Lambda\cdot \Z^n$ for a \emph{generator matrix} $M_\Lambda\in GL_n(\R)$ of $\Lambda$, whose columns are the basis vectors in $B$.  We define $\vol(\Lambda) = |\det(M_\Lambda)|$.

Given $x\in \R^n$ we define $N(x) = |x_1\cdots x_n|$.  The \emph{sphere-packing radius} $\rho(\Lambda)$ and \emph{minimum product distance} $N(\Lambda)$ of a lattice $\Lambda\subset\R^n$ are defined to be
\begin{equation}
\rho(\Lambda) := \frac{1}{2}\min_{0\neq x\in \Lambda}||x|| \quad\text{and}\quad N(\Lambda):= \min_{0\neq x\in \Lambda}N(x)
\end{equation}
Loosely speaking, we are interested in constructing infinite families of lattices $\Lambda$ such that both $\rho(\Lambda)$ and $N(\Lambda)$ are large.  To normalize things properly, we scale $\Lambda$ by a positive constant so that $\vol(\Lambda) = 1$ when computing $\rho(\Lambda)$ or $N(\Lambda)$.

\subsection{Moduli Spaces and Twists of Lattices}

Scaling a lattice $\Lambda$ by a positive constant and negating a basis vector if necessary we may assume $M_\Lambda\in SL_n(\R)$.  A change of basis corresponds to right multiplication on $M_\Lambda$ by an element of $SL_n(\Z)$, hence the space of all lattices in $\R^n$ is
\begin{equation}
\mathcal{L}_n := SL_n(\R)/SL_n(\Z).
\end{equation}
For $\Lambda\in \mathcal{L}_n$ its \emph{similarity class} is the orbit $[\Lambda]:= SO_n(\R)\cdot \Lambda\subset \mathcal{L}_n$.  The space of all lattices up to similarity is
\begin{equation}
\mathcal{S}_n := SO_n(\R)\backslash SL_n(\R)/SL_n(\Z).
\end{equation}
Note that $\rho(\Lambda)$ is a well-defined function $\mathcal{S}_n\rightarrow \R_{>0}$.  

The \emph{diagonal group} $\mathcal{A}_n\subset SL_n(\R)$ is defined to be
\begin{equation}
\mathcal{A}_n = \left\{\begin{bmatrix}
\alpha_1 & &  \\
& \ddots &  \\
& & \alpha_n
\end{bmatrix}:\ \alpha_i>0\quad \text{and}\quad \prod_{i = 1}^n \alpha_i = 1\right\}
\end{equation}
Given any $\Lambda\in \mathcal{L}_n$ the orbit $\mathcal{A}_n\cdot \Lambda$ is a submanifold of $\mathcal{L}_n$.  If $T\in \mathcal{A}_n$ we will call $T\cdot \Lambda$ a \emph{twist} of $\Lambda$.  Since $SO_n(\R)\cap \mathcal{A}_n=\{1\}$, the orbit $\mathcal{A}_n\cdot \Lambda$ also defines a submanifold of $\mathcal{S}_n$:
\begin{equation}
\gamma(\Lambda) = [\mathcal{A}_n\cdot\Lambda] = \{[T\cdot\Lambda]\in \mathcal{S}_n: T\in \mathcal{A}_n\}
\end{equation}
We are interested in understanding the intersection of $\gamma(\Lambda)$ with the \emph{well-rounded locus}, and computing the points of intersection explicitly.  

\subsection{Well-Rounded Lattices}

We now recall some basic facts about well-rounded lattices, especially those in $\R^2$.  For a general reference we recommend the book \cite{martinet_book}, as well as the articles \cite{lenny1,lenny2}.  Given a lattice $\Lambda$, define its set of minimal vectors to be
\begin{equation}
S(\Lambda) = \{x\in \Lambda: ||x|| = 2\rho(\Lambda)\}.
\end{equation}
A lattice $\Lambda\subset\R^n$ is \emph{well-rounded} if $\text{span}_\R(S(\Lambda)) = \R^n.$ As this property is invariant under the action of $SO_n(\R)$, the set of similarity classes of well-rounded lattices defines a sub-manifold of $\mathcal{S}_n$, which we denote by $\mathcal{W}_n$:
\begin{equation}
\mathcal{W}_n := \{[\Lambda]\in \mathcal{S}_n: \Lambda \text{ is well-rounded}\} \subset \mathcal{S}_n
\end{equation}
We call $\mathcal{W}_n$ the \emph{well-rounded locus}.

For a lattice $\Lambda\subset\R^2$ with basis $B = \{x,y\}$, we let $\theta_B\in (0,\pi)$ be the angle between $x$ and $y$.  It is not hard to show that $\Lambda$ is well-rounded if and only if it has a basis $B$ such that 
\begin{itemize}
\item[(i)] $||x||^2 = ||y||^2$, and
\item[(ii)] $\left|\cos\theta_B\right| = \left|\frac{\langle x, y\rangle }{||x||\cdot ||y||}\right|\leq1/2$, or equivalently $\theta_B\in [\pi/3,2\pi/3]$
\end{itemize}
We call a basis $B$ satisfying (i) and (ii) a \emph{minimal basis}, and the corresponding $\theta_B$ the $\emph{minimal angle}$ of $\Lambda$.  If $\Lambda$ and $\Lambda'$ are two well-rounded lattices with minimal bases $B$ and $B'$, then $[\Lambda] = [\Lambda']$ in $\mathcal{W}_2$ if and only if $|\cos\theta_B| = |\cos\theta_{B'}|$.  The sphere-packing radius of a well-rounded lattice $\Lambda$ is completely determined by the value of $|\cos\theta_B|$:
\begin{equation}
\rho(\Lambda) = \frac{1/2}{(1-|\cos\theta_B|^2)^{1/4}},\quad \vol(\Lambda) = 1.
\end{equation}
As such, we will use $|\cos\theta_B|$ as a proxy for $\rho(\Lambda)$ in comparing the sphere-packing radii of well-rounded lattices in $\R^2$.  If $\Lambda\subset\R^2$ is well-rounded and $\cos\theta_B = 0$ then $\Lambda$ is \emph{orthogonal}, and if $|\cos\theta_B| = 1/2$ then $\Lambda$ is \emph{hexagonal}.  



\subsection{Number Fields and Ideal Lattices}

Throughout, we let $K$ be a real quadratic field.  We adopt the following standard notation: $\mathcal{O}_K$ is the ring of integers of $K$, $\mathcal{I}$ is an ideal of $\mathcal{O}_K$, $\Delta_K$ is the discriminant of $K$ over $\Q$, $N:K\rightarrow \Q$ and $Tr:K\rightarrow \Q$ are the norm and trace maps, respectively, $N(\mathcal{I}) = |\mathcal{O}_K/\mathcal{I}|$ is the norm of an ideal $\mathcal{I}$, and $\Gal(K/\Q)$ is the Galois group of $K/\Q$ with non-trivial element $\sigma$.  For any $x\in K$, we let $\bar{x} = \sigma(x)$ denote its Galois conjugate.  If $K = \Q(\sqrt{D})$ for square-free $D>0$, then
\begin{equation}
\omega = \left\{\begin{array}{cl}
\frac{1+\sqrt{D}}{2} & D\equiv 1 \text{ mod $4$} \\
\sqrt{D} & D\not\equiv 1 \text{ mod $4$}
\end{array}\right.
\end{equation}
so that $\mathcal{O}_K = \Z[\omega]$.  We refer to $\{1,\omega\}$ as the \emph{canonical basis} of $\mathcal{O}_K$.  More generally, if $\mathcal{I}$ is any non-zero ideal of $\mathcal{O}_K$, it has a \emph{canonical basis} of the form
\begin{equation}\label{ideal_can_basis}
\{a, b + d\omega\},\ a,b,d\in \Z_{\geq 0} \text{ such that } b<a,\ d|a,\ d|b,\ \text{and } ad|N(b+d\omega).
\end{equation}

We are mostly interested in ideal lattices.  Let $\{\tau_1,\tau_2\}$ be the embeddings of $K$ into $\R$, and let $\psi:K\hookrightarrow \R^2$ be the canonical embedding, given by $\psi(x) = (\tau_1(x),\tau_2(x))$.  We generally identify $K$ with a subfield of $\R$ in a standard way, in which case $\tau_1 = \text{id}$ and $\tau_2 = \sigma$, so for example $\psi(\sqrt{D}) = (\sqrt{D},-\sqrt{D})$.  For an ideal $\mathcal{I}\subseteq\mathcal{O}_K$, define the \emph{ideal lattice} $\Lambda_\mathcal{I} = \psi(\mathcal{I})$, which has $\vol(\Lambda_{\mathcal{I}}) = N(\mathcal{I})\sqrt{\Delta_K}$, see \cite[Proposition 2.1]{eva_garden}.  If $\mathcal{I} = \mathcal{O}_K$, we define $\Lambda_K = \Lambda_{\mathcal{O}_K}$.  More generally, one can consider ideals in non-maximal orders of $K$, but for the sake of simplicity, and because we could not improve on our constructions by considering non-maximal orders, we choose to work with ideals in $\mathcal{O}_K$.

\subsection{Summary of Main Results}

We now summarize in more detail the structure and main results of the current paper.  In Section \ref{wr_twists_any_lattice}, we define the notion of a basis $B$ of a lattice $\Lambda\subset\R^2$ to be \emph{good for twisting}, by which we mean there exists a twisting matrix $T_\alpha$ so that $T_\alpha \Lambda$ is well-rounded and $T_\alpha B$ is a minimal basis of this well-rounded twist.  Theorem \ref{best_wr_cond} provides necessary and sufficient conditions for a basis $B$ to be good for twisting, and gives a natural bijection between the set $w(\Lambda)$ and bases $B$ which are good for twisting modulo a certain equivalence relation.

In Section \ref{wr_twists_ideal_lattices}, we study ideal lattices coming from real quadratic fields, and applying Theorem \ref{best_wr_cond} we arrive at Theorem \ref{norm_cond}, a necessary and sufficient condition for a basis $\{x,y\}$ of an ideal $\mathcal{I}$ to be good for twisting, in terms of $N(x)$ and $N(y)$.  As a corollary of Theorem \ref{norm_cond}, we show that if $\{x,y\}$ is a good basis then $|N(x)|$ and $|N(y)|$ are bounded by $N(\mathcal{I})\sqrt{\Delta_K/3}$.

In Section \ref{computing_good_bases}, we focus on explicit computation of all well-rounded twists of $\Lambda_\mathcal{I}$.  The main result is Theorem \ref{thm2}, which shows that given $x\in\mathcal{I}$, there exist at most two bases containing $x$ which are good for twisting, up to equivalence.  The proof of this theorem also yields an explicit algorithm for computing these bases, and therefore for computing the set $w(\Lambda_\mathcal{I})$.  We also obtain an upper bound on the number $|w(\Lambda_{\mathcal{I}})|$ of well-rounded twists of $\Lambda_\mathcal{I}$.

In Section \ref{ortho_family}, we digress slightly to study an infinite family of lattices $\Lambda_K$ such that the only well-rounded twist of $\Lambda_K$ is the orthogonal lattice.  Theorem \ref{big_theorem} shows this condition is equivalent to a lower bound on the regulator $R_K$ being met with equality.

Lastly, in Section \ref{markoff_lattices}, we apply recent results of A.\ Srinivasan \cite{srinivasan} to show that among all ideal lattices from real quadratic fields, the one with maximal $N(\Lambda)$ is $\Lambda = \Lambda_K$ where $K = \Q(\sqrt{5})$, which for dimension $n = 2$ answers a question raised in \cite[Section III]{oggier_belfiore_bayer}.  We then consider ideal lattices $\Lambda_c$ arising from Markoff numbers, an infinite family of ideal lattices which have $N(\Lambda_c)>1/3$.  Using the tools of Section \ref{computing_good_bases}, we construct four infinite sub-families of Markoff lattices, arising from Fibonacci and Pell numbers, such that $\cos\theta$ approaches $0$, $(6-4\sqrt{5})/11$, $(3-\sqrt{2})/7$, and $(15-11\sqrt{2})/17$ as $c\rightarrow \infty$, respectively.  

\section{Well-Rounded Twists of Planar Lattices}\label{wr_twists_any_lattice}

\subsection{Geodesics in the Upper-Half Plane}

We begin with a picture in the upper-half plane $\mathcal{H} = \{(x,y)\in \R^2: y > 0\}$.  Given a lattice $\Lambda\subset\R^2$, we can rotate, scale, and change bases so that the first basis vector is $[1\ 0]^T$ and the other is in $\mathcal{H}$.  Modding out by the action of $SL_2(\Z)$ on $\mathcal{H}$ via fractional linear transformations, we arrive at the classical identification of $\mathcal{S}_2$ with the \emph{fundamental domain} $\mathcal{F}$:
\begin{equation}
\mathcal{S}_2 = \mathcal{F} = \left\{(x,y)\in \mathcal{H}: -\frac{1}{2}<x\leq \frac{1}{2},\ x^2+y^2\geq 1\right\}
\end{equation}
The well-rounded locus $\mathcal{W}_2$ sits inside $\mathcal{F}$ as the bottom arc:
\begin{equation}
\mathcal{W}_2 = \{(x,y)\in \mathcal{F}: x^2 + y^2 = 1, x\geq 0\} = \{(x,y)\in \mathcal{F}: x^2 + y^2 = 1\}/\sim
\end{equation}
where the equivalence relation $\sim$ identifies $(x,y)$ with $(-x,y)$.  This equivalence relation can be realized in terms of lattices as the change of basis $\{u,v\}\mapsto\{u,-v\}$ followed by a rotation which sends $-v$ to $(1,0)$.

Consider the diagonal group $\mathcal{A}_2$ 
\begin{equation}
\mathcal{A}_2 = \left\{T_\alpha = \begin{bmatrix} \alpha & 0 \\ 0 & 1/\alpha \end{bmatrix}: \alpha > 0\right\}
\end{equation}
and let $\Lambda\in \mathcal{L}_2$ be a lattice.  We define
\begin{equation}\label{geodesic_and_wr}
\gamma(\Lambda) = [\mathcal{A}_2\cdot \Lambda]\subset \mathcal{F} \quad\text{and}\quad
w(\Lambda) = \gamma(\Lambda)\cap \mathcal{W}_2
\end{equation}
Hence $\gamma(\Lambda)$ is the geodesic in $\mathcal{F}$ corresponding to the similarity classes of lattices in the orbit $\mathcal{A}_2\cdot \Lambda$, and $w(\Lambda)$ is the set of all well-rounded twists of $\Lambda$.

We can describe $\gamma(\Lambda)$ explicitly as follows.  Fix a basis $B = \{[a\ c]^T, [b\ d]^T\}$ of $\Lambda$ such that $ad-bc >0$, and let $T_\alpha$ be a twisting matrix.  Rotating and scaling $T_\alpha B$ so that the first basis vector becomes $[1\ 0]^T$, we get the following generator matrix for a lattice in the same similarity class as $T_\alpha \Lambda$:
\begin{equation}
\left[\begin{matrix} 1 \\ 0 \end{matrix}\ \ \tau(\Lambda,\alpha)\right]
\quad \text{where}\quad
\tau(\Lambda,\alpha) = \frac{1}{a^2\alpha^4+c^2}\begin{bmatrix}
ab\alpha^4 + cd \\
\alpha^2(ad - bc)
\end{bmatrix}.
\end{equation}
Now $\tau(\Lambda,\alpha)$ traces out a curve in the upper half plane as $\alpha\in \R_{>0}$ varies, and modding out by the action of $SL_2(\Z)$ we obtain $\gamma(\Lambda)=\tau(\Lambda,\alpha)/SL_2(\Z)\subset \mathcal{F}$.

\begin{figure}
\centering
\includegraphics[width=.3\textwidth]{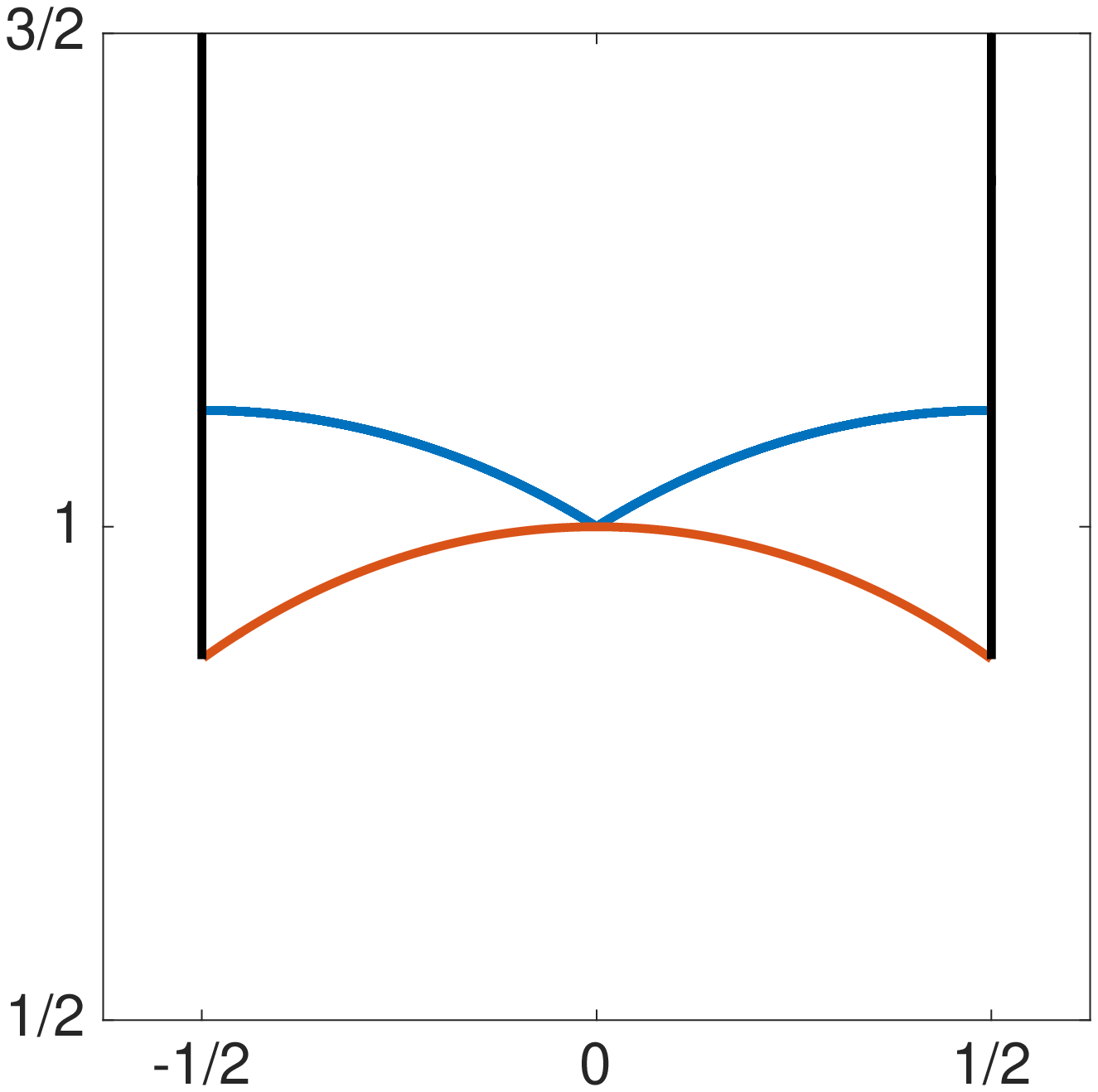}\hfill
\includegraphics[width=.3\textwidth]{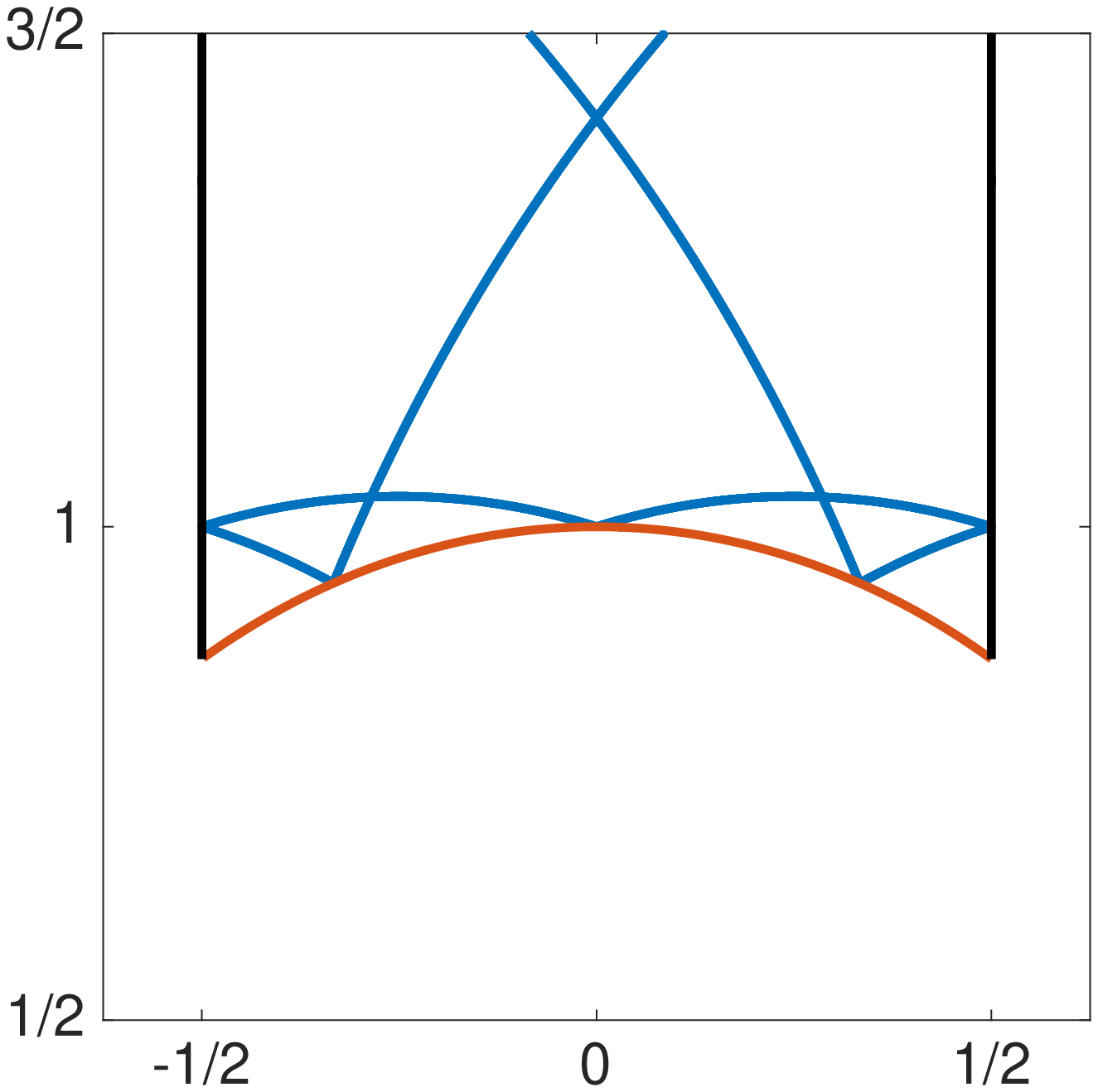}\hfill
\includegraphics[width=.3\textwidth]{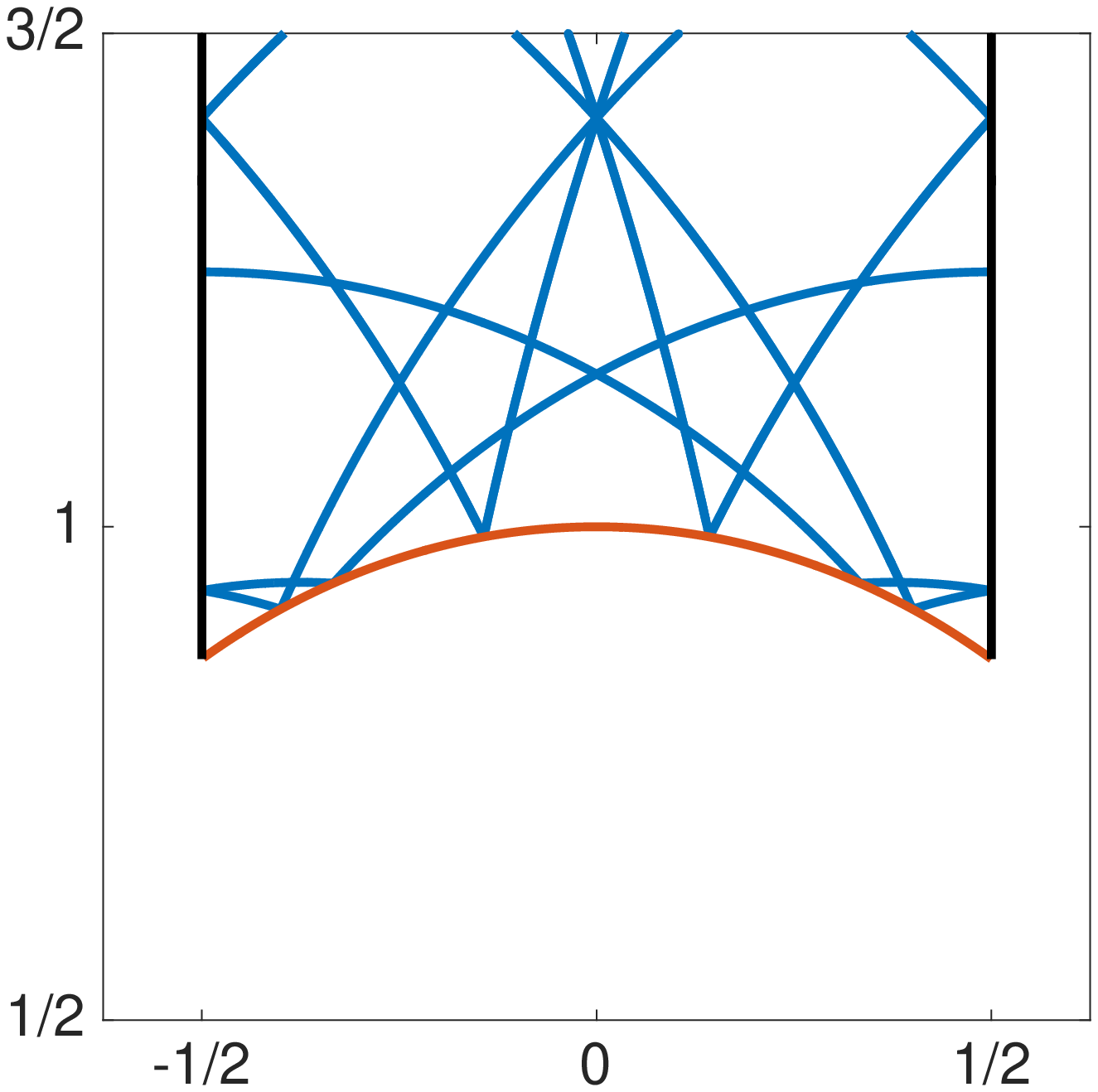}
\caption{A plot of the geodesic $\gamma(\Lambda_K)$, for $K = \Q(\sqrt{D})$, for $D = 5$ (left), $D = 17$ (center), and $D = 57$ (right).  The blue curve is $\gamma(\Lambda_K)$, the orange curve $\mathcal{W}_2$, and their intersection the set $w(\Lambda_K)$ of well-rounded twists of $\Lambda_K$.}\label{TL_orbit}
\end{figure}

In Fig.\ \ref{TL_orbit} we depict the curve $\gamma(\Lambda_K)$ for $K =\Q(\sqrt{D})$ for $D = 5$ (left), $D = 17$ (center), and $D = 57$ (right).  Here the blue curve is $\gamma(\Lambda_K)$ and the orange curve is $\mathcal{W}_2$.  The picture suggests that $|w(\Lambda)| = 1$, $2$, and $3$, respectively.  In what follows, we show how to calculate these well-rounded twists explicitly.

\begin{example}\label{first_example}

This example is familiar from the theory of lattice coding for Rayleigh fading channels  \cite[Section 7.1]{oggier_belfiore}, and also appears in the construction of the Golden Code \cite{golden_code} used in multiple-antenna communications.

Let $K = \Q(\sqrt{5})$ and let $\mathcal{I} = \mathcal{O}_K$.  Fig.\ \ref{TL_orbit} suggests that $\Lambda_K$ has exactly one well-rounded twist, namely the orthogonal lattice.  To construct this twist explicitly, consider the basis $B = \{[1\ 1]^T, [\omega,\bar{\omega}]^T\}$ of the lattice $\Lambda_K$.  Letting $\alpha = \left(-\bar{\omega}/\omega\right)^{1/4}$, one checks that the basis $T_\alpha B$ of $T_\alpha \Lambda$ consists of two equal-length orthogonal vectors.  
\end{example}

\subsection{A General Condition for Well-Rounded Twists}

Throughout this section, we denote by $\Lambda$ a lattice in $\R^2$, and by $B$ a basis of $\Lambda$:
\begin{equation}
B = \{x,y\} = \{[a\ c]^T, [b\ d]^T\}.
\end{equation}
We define the polynomial
\begin{equation}\label{fdefn}
F(B) = (ac)^2 + abcd + (bd)^2 - \vol(\Lambda)^2/4.
\end{equation}
We put an equivalence relation on bases of $\Lambda$ according to the value of $F(B)$:
\begin{equation}
B\sim B' \Leftrightarrow F(B) = F(B').
\end{equation}
For example, if $B = \{x,y\}$ then all four bases $\{\pm x,\pm y\}$ are equivalent.  

The goal of this section is to establish a natural bijection between $w(\Lambda)$ and equivalence classes of bases satisfying $F(B)\leq 0$.  We begin by focusing on the first condition (i) in our characterization of planar well-rounded lattices.

\begin{definition}\label{twistable_defn}
We will say that $B$ is \emph{twistable} if there exists $T_\alpha\in \mathcal{A}_2$ such that $||T_\alpha x||^2 = ||T_\alpha y||^2$.
\end{definition}

\begin{prop}\label{twisting_element}
Let $\beta = \frac{d^2-c^2}{a^2-b^2}$.  Then $B$ is twistable if and only if $\beta>0$.  If this is the case, then $T_\alpha$ is unique and $\alpha = \beta^{1/4}$.
\end{prop}
\begin{proof} Writing out the equation $||T_\alpha x||^2 = ||T_\alpha y||^2$ with $\alpha$ as a variable and clearing denominators yields an equation in $\alpha^4$, and solving this equation gives $\alpha^4 = \frac{d^2-c^2}{a^2-b^2}$.  The result follows.
\end{proof}

\begin{definition}
If $B$ is twistable, we let $\theta_{T_\alpha B}$ denote the angle between the resulting basis vectors $T_{\alpha} B = \{T_{\alpha} x, T_{\alpha} y\}$ of the twisted lattice $T_{\alpha} \Lambda$.  To emphasize that $\alpha$ is a function of the basis $B$, we will sometimes write $\alpha(B)$ for $\alpha$.
\end{definition}

\begin{prop}\label{cos_of_twist}
For a twistable basis $B$ with twisting matrix $T_\alpha$, we have
\begin{equation}
\cos \theta_{T_\alpha B} = \frac{ac+bd}{ad+bc}.
\end{equation}
\end{prop}
\begin{proof}
Using the fact that $||T_\alpha x|| = ||T_\alpha y||$, we compute that
\begin{align}
\cos \theta_{T_\alpha B} &= \frac{\langle T_\alpha x, T_\alpha y\rangle}{||T_\alpha x||\cdot ||T_\alpha y||} = \frac{\alpha^4 ab + cd}{\alpha^4 a^2 + c^2} \\
&= \frac{abd^2 - abc^2 + a^2cd -b^2cd}{a^2d^2-b^2c^2} = \frac{ac+bd}{ad+bc}
\end{align}
which is what was claimed.
\end{proof}

\begin{prop}\label{good_implies_twistable}
If $\kappa(B) := \left|\frac{ac+bd}{ad+bc}\right|\leq 1/2$, then $\beta = \frac{d^2-c^2}{a^2-b^2}>0$.  
\end{prop}
\begin{proof}
First note that we must have $a\neq \pm b$ and $c\neq \pm d$, else $\left|\frac{ac+bd}{ad+bc}\right| = 1$ which contradicts our assumption.  Squaring both sides of $\left|\frac{ac+bd}{ad+bc}\right|\leq 1/2$ and simplifying gives
\begin{equation}
0\leq (a^2-b^2)(d^2-c^2) - 3(ac-bd)^2.
\end{equation}
Dividing by the positive quantity $(a^2-b^2)^2$ gives
\begin{equation}
0\leq 3\frac{(ac-bd)^2}{(a^2-b^2)^2}\leq \frac{d^2-c^2}{a^2-b^2} = \beta
\end{equation}
but $\beta\neq 0$ since $c\neq \pm d$, which finishes the proof.
\end{proof}

Thus if $\Lambda$ has a basis $B$ such that $\kappa(B)\leq 1/2$ and we set $\alpha = \beta^{1/4}$ (a well-defined positive real number by Proposition \ref{good_implies_twistable}), the twisted lattice $T_\alpha \Lambda$ is well-rounded with minimal basis $T_\alpha B$ and minimal angle $\theta_{T_\alpha B}$.  Conversely, if $B$ is twistable and the resulting twist is well-rounded with minimal basis $T_\alpha B$, we must have $\kappa(B)\leq 1/2$ by Proposition \ref{cos_of_twist}.  This prompts the following definition.

\begin{definition}\label{good_defn}
A basis $B$ of $\Lambda$ is \emph{good for twisting}, or simply a \emph{good basis}, if there exists a twisting matrix $T_\alpha\in\mathcal{A}_2$ such that $T_\alpha\Lambda$ is well-rounded with minimal basis $T_\alpha B$.  Equivalently, $B$ is good for twisting if and only if $\kappa(B) \leq 1/2$.
\end{definition}

Thus the property of being good for twisting as in Definition \ref{good_defn} is more demanding that that of merely being twistable as in Definition \ref{twistable_defn}.  If $B$ is  twistable but not \emph{good} for twisting, then $T_\alpha B$ need not be a \emph{minimal} basis of $T_\alpha \Lambda$, and moreover $T_\alpha \Lambda$ need not even be well-rounded.

\begin{theorem}\label{best_wr_cond}
There is a natural bijection between the set $w(\Lambda)$ of equivalence classes of well-rounded twists of $\Lambda$, and equivalence classes of good bases.  More specifically,
\begin{itemize}
\item[(i)] A basis $B$ of $\Lambda$ is good for twisting if and only if $F(B) \leq 0$.  
\item[(ii)] If $B$ and $B'$ are both good for twisting with associated twisting matrices $T_\alpha$ and $T_{\alpha'}$, then $[T_\alpha\Lambda] = [T_{\alpha'}\Lambda]$ in $\mathcal{W}_2$ if and only if $F(B) = F(B')$.
\item[(iii)] If $B$ is good for twisting, then
\begin{equation}
(ac)^2\leq \vol(\Lambda)^2/3\quad\text{and}\quad (bd)^2\leq \vol(\Lambda)^2/3.
\end{equation}
\end{itemize}
\end{theorem}
\begin{proof} Part (i) is seen by squaring both sides of $\kappa(B)\leq 1/2$ (which is an invertible process, since both sides are positive) and clearing denominators to obtain $(ac+bd)^2 \leq \frac{1}{4}(ad+bc)^2$.  Now expand out and simplify using the formula $(ad-bc)^2 = \vol(\Lambda)^2$ to get the result.

The first condition in part (ii) is equivalent to $\kappa(B) = \kappa(B')$, which is in turn equivalent to $F(B) = F(B')$ by the proof of the previous proposition.

For part (iii), consider the function $f:\R^2\rightarrow \R$ defined by $f(X,Y) = X^2 + XY + Y^2$.  Setting $Y = \gamma X$ for some real number $\gamma$ gives $f(X,\gamma X) = (\gamma^2 + \gamma + 1)X^2$, and a simple calculus exercise shows that the function $g:\R\rightarrow \R$ given by $g(\gamma) = \gamma^2 + \gamma + 1$ has a single minimum at $\gamma = -1/2$ with value $g(-1/2) = 3/4$.  It follows that $f(X,Y)\geq \frac{3}{4}X^2$.  Setting $X = ac$ and $Y = bd$ and using condition (iii) yields
\begin{equation}
(ac)^2\leq \frac{4}{3}\left((ac)^2 + abcd + (bd)^2\right)\leq \vol(\Lambda)^2/3
\end{equation}
as claimed.  By symmetry, we also have $(bd)^2\leq \vol(\Lambda)^2/3$, which proves the third part of the theorem.
\end{proof}

\section{Well-Rounded Twists of Planar Ideal Lattices}\label{wr_twists_ideal_lattices}

\subsection{Good Bases of Ideals}

Let $\mathcal{I}$ be an ideal in the ring of integers $\mathcal{O}_K$ of a real quadratic field $K = \Q(\sqrt{D})$.  Clearly, the canonical embedding $\psi:K\hookrightarrow\R^2$ determines a bijection between $\Z$-bases $B = \{x,y\}$ of $\mathcal{I}$ and $\Z$-bases $\psi(B) = \{\psi(x),\psi(y)\}$ of $\Lambda_\mathcal{I}$.  In what follows, we will simply say ``basis'' instead of ``$\Z$-basis''.

\begin{definition}
A basis $B$ of $\mathcal{I}$ is \emph{twistable} if the basis $\psi(B)$ of $\Lambda_\mathcal{I}$ is twistable.  Similarly, $B$ is \emph{good for twisting} if the basis $\psi(B)$ of $\Lambda_\mathcal{I}$ is good for twisting.  If $B$ is a twistable basis of $\mathcal{I}$, then we define $\theta_{T_\alpha B}$ to be the angle between the basis vectors in $T_\alpha B:=T_\alpha\psi(B)$, where $\alpha = \alpha(B) = \alpha(\psi(B))$ is as in Proposition \ref{twisting_element}.
\end{definition}

\begin{prop}\label{cos}
Let $B=\{x,y\}$ be a twistable basis of $\mathcal{I}$.  Then
\begin{equation}\label{coseqn}
\cos \theta_{T_\alpha B}  = \frac{N(x)+N(y)}{Tr(x\bar{y})}.
\end{equation}
\end{prop}
\begin{proof}
This is a special case of Proposition \ref{cos_of_twist} in the case where $\Lambda = \Lambda_\mathcal{I}$ is an ideal lattice.
\end{proof}

For any basis $B = \{x,y\}$ of an ideal $\mathcal{I}$, the polynomial $F(B)$ of (\ref{fdefn}) is of the form
\begin{equation}\label{useful_poly}
F(B) = F(x,y)=N(x)^2 + N(x)N(y) + N(y)^2 - N(\mathcal{I})^2\Delta_K/4.
\end{equation}
When $\Lambda = \Lambda_\mathcal{I}$ is an ideal lattice, Theorem \ref{best_wr_cond} takes the following form.

\begin{theorem}\label{norm_cond}
Let $\mathcal{I}$ be an ideal with basis $B = \{x,y\}$.  Then
\begin{itemize}
\item[(i)] $B$ is good for twisting if and only if $F(x,y)\leq 0$, in which case the twisting matrix $T_\alpha$ is given by $\alpha = ((\bar{y}^2-\bar{x}^2)(x^2-y^2))^{1/4}$.
\item[(ii)] If $B=\{x,y\}$ and $B'=\{x',y'\}$ are two good bases of $\mathcal{I}$ with corresponding twisting elements $\alpha$ and $\alpha'$, then $[T_\alpha \Lambda_\mathcal{I}] = [T_{\alpha'}\Lambda_{\mathcal{I}}]$ in $\mathcal{W}_2$ if and only if $F(x,y) = F(x',y')$.
\item[(iii)] If $B$ is good for twisting, then
\begin{equation}
N(x)^2\leq N(\mathcal{I})^2\Delta_K/3 \quad \text{and}\quad N(y)^2\leq N(\mathcal{I})^2\Delta_K/3.
\end{equation}
\end{itemize}
\end{theorem}
\begin{proof}
This is a special case of Theorem \ref{best_wr_cond} when $\Lambda = \Lambda_\mathcal{I}$ is an ideal lattice.
\end{proof}

Thus explicitly describing the set $w(\Lambda_\mathcal{I})$ reduces to studying bases $B = \{x,y\}$ such that $F(B)\leq 0$.  Given a good basis $B$ of $\mathcal{I}$, the equivalence class of the resulting well-rounded twist is then easily calculable using Proposition \ref{cos}.

By construction, we have a factorization $F(x,y) = F_1(x,y) F_2(x,y)$, where
\begin{align}
F_1(x,y) &= N(x) + N(y) + Tr(x\bar{y})/2 \label{f1} \\
F_2(x,y) &= N(x) + N(y) - Tr(x\bar{y})/2 \label{f2} 
\end{align}
and thus for a good basis $B = \{x,y\}$ of $\Lambda$, we have the expressions
\begin{equation}\label{cos_of_ideal}
\cos\theta_{T_\alpha B} = \frac{F_1(x,y) - Tr(x\bar{y})/2}{Tr(x\bar{y})} = \frac{F_2(x,y) + Tr(x\bar{y})/2}{Tr(x\bar{y})}
\end{equation}
which allow us to classify when the orthogonal and hexagonal lattices appear as twists of $\Lambda_\mathcal{I}$.
\begin{prop}\label{orth_and_hex}
Let $\mathcal{I}\subseteq\mathcal{O}_K$ be an ideal.
\begin{itemize}
\item[(i)] The orthogonal lattice is a twist of $\Lambda_\mathcal{I}$ if and only if $\mathcal{I}$ has a basis $B = \{x,y\}$ such that $N(x) + N(y) = 0$.
\item[(ii)] The hexagonal lattice is a twist of $\Lambda_\mathcal{I}$ if and only if $\mathcal{I}$ has a basis $B = \{x,y\}$ such that $F(x,y) = 0$.
\end{itemize}
\end{prop}
\begin{proof}
For part (i), we must show that any basis $\{x,y\}$ that satisfies $N(x)+N(y)=0$ must also satisfy $F(x,y)\leq 0$, then the result follows immediately from equation (\ref{coseqn}) and Theorem \ref{norm_cond}.  To prove this, we use the identity
\begin{equation}
\vol(\Lambda_\mathcal{I})^2 = Tr(x\bar{y})^2 - 4N(x)N(y)
\end{equation}
If $N(y) = -N(x)$, this implies that $4N(x)^2 = \vol(\Lambda_\mathcal{I})^2 - Tr(x\bar{y})^2\leq \vol(\Lambda_\mathcal{I})^2$, so that $F(x,y) = N(x)^2- \vol(\Lambda_\mathcal{I})^2/4\leq 0$.

Part (ii) follows from (\ref{cos_of_ideal}), since $F(x,y) = 0$ if and only if one of the $F_i(x,y) = 0$ if and only if $\cos \theta_{T_\alpha B} = \pm 1/2$.
\end{proof}

\begin{prop}\label{orthog12}
Suppose that the orthogonal lattice is a twist of $\Lambda_K$.  Then $D\equiv 1,2\pmod 4$.
\end{prop}
\begin{proof}
Suppose that $D\not\equiv1\pmod 4$.  A good basis $\{x,y\}$ will give an orthogonal twist of $\Lambda_K$ if and only if $N(x) + N(y) = 0$.  Expressing $x$ and $y$ as $x = a + c\sqrt{D}$ and $y = b + d\sqrt{D}$, we have $a^2 + b^2 = (c^2 + d^2)D$.  Since $x,y$ is a basis of $\mathcal{O}_K$, we must have $ad-bc=\pm1$ and in particular $\gcd(a,b) = 1$.  If $p|D$ then $a^2 + b^2 \equiv 0\pmod p$, and furthermore $p$ does not divide either $a$ or $b$.  Thus $(a/b)^2 \equiv -1\pmod p$, hence $-1$ is a quadratic residue mod $p$.  This forces $p = 2$ or $p\equiv 1\pmod 4$, hence $D\equiv 1,2\pmod 4$.
\end{proof}

It is natural to ask what, if any, connection there is between twisting an ideal lattice by an element of the diagonal group as we are doing, and twisting by a matrix of the form
\begin{equation}\label{alg_twist}
S_\gamma = \begin{bmatrix}
\sqrt{\gamma} & 0 \\
0 & \sqrt{\bar{\gamma}}
\end{bmatrix},\quad \gamma \in K,\ \text{$\gamma$ totally positive}
\end{equation}
as is often done, for example \cite{oggier_belfiore,eva_garden,eva_suarez} and other work by  Bayer-Fluckinger et al.  Essentially, for the purposes of understanding the set $w(\Lambda_\mathcal{I})$, these two notions of twisting are equivalent.

\begin{prop}
Let $\Lambda_\mathcal{I}$ be an ideal lattice.  Then the set of well-rounded lattices of the form $T_\alpha\cdot\Lambda_\mathcal{I}$ for $T_\alpha\in\mathcal{A}_2$ is in bijection with the set of well-rounded lattices of the form $S_\gamma\cdot\Lambda_\mathcal{I}$ for $S_\gamma$ as in (\ref{alg_twist}). 
\end{prop}
\begin{proof}
Any twisting matrix of the form (\ref{alg_twist}) can be scaled to be in the diagonal group, so any well-rounded twist of the form $S_\gamma\cdot \Lambda_\mathcal{I}$ clearly gives one of the form $T_\alpha \cdot \Lambda$.  To show the other inclusion, suppose that $T_\alpha$ is the twisting matrix associated with a good basis $\{x,y\}$ of $\mathcal{I}$, so that $\alpha = ((\bar{y}^2-\bar{x}^2)/(x^2-y^2))^{1/4}$.  Suppose first that $\bar{y}^2 - \bar{x}^2>0$ and $x^2 - y^2>0$.  Define $\gamma = (\bar{y}^2-\bar{x}^2)/\sqrt{D}$ which is totally positive and satisfies $\alpha = (\gamma/\bar{\gamma})^{1/4}$.  If $S_\gamma$ is as in (\ref{alg_twist}), then one computes easily that $S_\gamma\cdot \Lambda = N(\gamma)^{1/4}T_\alpha\cdot\Lambda$, hence these two lattices define the same element of $\mathcal{W}_2$.  The proof in the case $\bar{y}^2 - \bar{x}^2<0$ and $x^2 - y^2<0$ is similar.
\end{proof}

\subsection{Good Bases of $\Lambda_\mathcal{I}$, Units, and Principal Ideals}

Our goal is to compute, up to equivalence, all of the well-rounded twists of a given ideal lattice $\Lambda_\mathcal{I}$.  As the next proposition shows, Proposition \ref{norm_cond} allows us to discard some obvious transformations of good bases to make this computational task more tractable.

\begin{prop}\label{units_and_galois}
Let $B = \{x,y\}$ be a basis of $\mathcal{I}$ which is good for twisting. Then $uB$ is also good for twisting for any unit $u\in\mathcal{O}_K$, and $B$ and $uB$ are equivalent bases.  Similarly, the basis $\sigma(B) = \{\bar{x},\bar{y}\}$ of $\sigma(\mathcal{I})$ is good for twisting, and $F(\sigma(B)) = F(B)$.  In particular if $\sigma(\mathcal{I}) = \mathcal{I}$, then $B$ and $\sigma(B)$ are equivalent bases.
\end{prop}
\begin{proof}
One calculates easily that $F(uB) = F(B)$ and $F(\sigma(B)) = F(B)$.  The statements now follow easily from Theorem \ref{norm_cond}.
\end{proof}

The above proposition suggests that our problem of describing $w(\Lambda_\mathcal{I})$ reduces to studying principal ideals within $\mathcal{I}$.  More precisely, if we fix an element $x\in \mathcal{I}$ and consider all good bases of the form $\{x,y\}$ of $\mathcal{I}$, then the set of resulting equivalence classes only depends on $(x)$.  Moreover, if $\mathcal{I}$ is fixed by the Galois action, for example when $\mathcal{I} = \mathcal{O}_K$, then the problem reduces to studying principal ideals within $\mathcal{I}$ up to Galois conjugation.

\begin{definition}
Let $x\in \mathcal{I}$.  We say that $x$ \emph{extends to a (good) basis} if there exists $y\in \mathcal{I}$ such that $\{x,y\}$ is a (good) basis of $\mathcal{I}$.
\end{definition}

\begin{prop}\label{extendbasis}
For any non-zero $x\in \mathcal{I}$, $x$ extends to a basis of $\mathcal{I}$ if and only if the ideal $(x)$ is not divisible by any ideal of the form $n\mathcal{O}_K$ for $n\in \Z$, $n\neq\pm1$.
\end{prop}
\begin{proof}
Let $\{u,v\}$ be any basis of $\mathcal{I}$ and write $x = au + cv$ for $a,c\in \Z$.  Then $x$ extends to a basis if and only if there exists $y = bu + dv\in \mathcal{I}$ such that $ad-bc = \pm1$.  This occurs if and only if $\gcd(a,c) = 1$, which is equivalent to $(x)$ not being divisible by any ideal of the form $n\mathcal{O}_K$ where $n\neq \pm1$.
\end{proof}

From Theorem \ref{norm_cond} and Propositions \ref{units_and_galois} and \ref{extendbasis} we see that the following strategy suffices to compute all well-rounded twists of a given ideal lattice $\Lambda_\mathcal{I}$:
\begin{itemize}
\item[(i)]  First, we list all principal ideals $(x)\subseteq\mathcal{I}$ such that (i) $N(x)^2\leq N(\mathcal{I})^2\Delta_K/3$ and (ii) $(x)$ is not divisible by any $(n)$ with $n\neq \pm1$.  
\item[(ii)] Second, for each such $(x)$, we pick a specific generator $x$ and solve $F(x,y)\leq0$ for all possible $y$ such that $\{x,y\}$ is a basis of $\mathcal{I}$.  
\end{itemize}
The result is a list of all bases of $\mathcal{I}$, up to equivalence, which are good for twisting.

\section{Computing All Well-Rounded Twists of $\Lambda_\mathcal{I}$}\label{computing_good_bases}

\subsection{Computing all good bases of an ideal $\mathcal{I}$}

We devote this subsection to explicit computations of good bases of ideals $\mathcal{I}$.  The next theorem provides us with simple bounds on how many good bases an element $x\in \mathcal{I}$ can extend to, and also gives us an effective algorithm to compute these bases.

\begin{theorem}\label{thm2}
Let $\mathcal{I}$ be an ideal in the ring of integers of a real quadratic field and let $x\in \mathcal{I}$ be such that $N(x)^2\leq N(\mathcal{I})^2\Delta_K/3$.  Then $x$ extends to at most two good bases of $\mathcal{I}$, up to equivalence.
\end{theorem}
\begin{proof}

Let us fix a basis $\{u,v\}$ of $\mathcal{I}$, for example, the canonical basis.  Express $x = au + cv$ in this basis for $a,c\in \Z$, and suppose $y = bu + dv$ for some $b,d\in \Z$ is such that $\{x,y\}$ is a good basis.  We will use the inequality $F(x,y)\leq 0$ and the equality $ad-bc = \pm1$ to solve for all possible $b$ and $d$.  We break the proof into two cases, according to whether $a = 0$ or $a\neq 0$.

\emph{Case 1: $a = 0$}.  If $a = 0$, then $ad-bc = \pm1$ implies that $b = \pm1$ and $c = \pm 1$, so $x = \pm v$ and $y = \pm u + dv$.  Then by possibly replacing $x$ by $-x$ and $y$ by $-y$, operations which do not change the equivalence class of the basis, we may assume without loss of generality that our basis $\{x,y\}$ is of the form $\{x,y\} = \{v,u+dv\}$.  We will show that there are at most two integers $d$ such that this is a good basis.

Let $f_i(d) = F_i(x,y)$, where $\{x,y\} = \{v,u+dv\}$ as above and $F_i(x,y)$ are as in (\ref{f1}) and (\ref{f2}).  By Theorem \ref{norm_cond}, we must find all $d$ such that the $f_i(d)$ have opposite signs, or such that at least one of them is zero.  The polynomials $f_i(d)$ are given by
\begin{equation}
f_i(d) = N(x)d^2  + (Tr(u\bar{v})\pm N(x))d + N(u) + N(x) \pm Tr(u\bar{v})/2
\end{equation}
where we choose the positive sign for $f_1(d)$.  The discriminant of these polynomials is the same, namely $\delta = \vol(\Lambda_\mathcal{I})^2 - 3N(x)^2\geq 0$, which is non-negative by assumption.  The roots $\beta_{i1}$ and $\beta_{i2}$ of $f_i(d)$ are given by
\begin{equation}
\beta_{11},\beta_{12} = \frac{-(Tr(u\bar{v})+N(x))\pm \sqrt{\delta}}{2N(x)},\quad \beta_{21},\beta_{22} = \frac{-(Tr(u\bar{v})-N(x))\pm \sqrt{\delta}}{2N(x)}
\end{equation}
where we choose the negative sign for $\beta_{i1}$.  Clearly, if $N(x)>0$ then $f_i(d)\leq0$ only on the interval $[\beta_{i1},\beta_{i2}]$.  Similarly if $N(x)<0$ then $f_i(d)\geq 0$ only on these intervals.

Now consider the intervals $J_1 = [\beta_{11},\beta_{21}]$ and $J_2 = [\beta_{12},\beta_{22}]$ which both have width one.  By the previous paragraph, if $f_1(d)\leq 0$ and $f_2(d)\geq 0$, or if $f_1(d)\geq 0$ and $f_2(d)\leq 0$, then we must have $d\in J_1\cup J_2$.  Since $J_1$ and $J_2$ both have width one, they contain at least one and at most two integers $d$.  If $J_i$ contains two integers $d$ then they are the endpoints of the interval, since it has width one.  But these endpoints are exactly the roots of $f_i(d)$, hence $F(x,y) = 0$ and both resulting well-rounded twists are hexagonal by Proposition \ref{orth_and_hex}.  Thus both intervals $J_1$ and $J_2$ each contribute at most one good basis of $\Lambda_\mathcal{I}$, up to similarity.

\emph{Case 2: $a\neq 0$}.    Suppose $x = au + cv$ with $a\neq 0$, and multiplying by $-1$ if necessary we may assume $a>0$.  We again explicitly calculate all $y = bu + dv\in \mathcal{I}$ such that $\{x,y\}$ is a good basis of $\mathcal{I}$, up to equivalence.  By possibly replacing $y$ with $-y$ we may assume $ad-bc = 1$, and solve for $d$ in terms of $b$ as $d = (1+bc)/a$.  Setting $f_i(b) = F_i(x,y)$, we wish to find all integers $b$ such that $d = (1+bc)/a\in \Z$, and that either $f_1(b)\leq 0$ and $f_2(b)\geq 0$, or $f_1(b)\geq 0$ and $f_2(b)\leq 0$.

The polynomials $f_i(b)$ are given by
\begin{equation}
\begin{aligned}
a^2 f_i(b) &= N(x)b^2 + (a(Tr(u\bar{v})\pm N(x)) + 2cN(v))b \\
&\phantom{aa}+ a^2(N(x)\pm Tr(u\bar{v})/2) + N(v)(1\pm ac)
\end{aligned}
\end{equation}
where we choose the positive sign for $f_1(b)$.  As in Case 1, we again find that the discriminant $\delta = \vol(\Lambda_\mathcal{I})^2 - 3N(x)^2\geq 0$ of both of these polynomials is the same.  The roots $\beta_{i1}$ and $\beta_{i2}$ of $f_i(b)$ are given by
\begin{align}
\beta_{11},\beta_{12} &= \frac{-(a(Tr(u\bar{v})+N(x))+2cN(v))\pm a\sqrt{\delta}}{2N(x)} \\
\beta_{21},\beta_{22} &= \frac{-(a(Tr(u\bar{v})-N(x))+2cN(v))\pm a\sqrt{\delta}}{2N(x)}
\end{align}
where we choose the negative sign for $\beta_{i1}$.  

As in Case 1, we define $J_1 = [\beta_{11},\beta_{21}]$ and $J_2 = [\beta_{12},\beta_{22}]$ and notice that $f_1$ and $f_2$ have opposite signs only inside the intervals $J_i$, and so any $b\in \Z$ such that $\{x,b+d\omega\}$ is a good basis must satisfy $b\in J_1\cup J_2$.  We will again show that each interval $J_i$ can contribute at most one good basis of $\Lambda_\mathcal{I}$, up to similarity.

The intervals $J_i$ are of width $a$, hence each contain at least $a$ and at most $a+1$ consecutive integers, the latter case occurring exactly when the endpoints themselves are integers.  First suppose that $c\neq 0$.  We wish to pick $b\in J_i$ such that $d = (1+bc)/a\in \Z$, or equivalently $b\equiv -c^{-1}\pmod a$.  As $\gcd(a,c) = 1$, the intervals $J_i$ contain at least one and at most two solutions to this congruence.  If the endpoints of the $J_i$ are not integers, we can therefore find exactly one integer $b\in J_i$ such that $d\in \Z$.  If the endpoints are integers, then again they both produce the hexagonal twist as in Case 1.  Thus each interval $J_i$ contributes at most one good basis, up to equivalence.  Now if $c = 0$, then without loss of generality we have $a = d = 1$.  In this case the intervals $J_i$ each have width $1$, and the same argument as above applies to prove that there exists either a unique $b\in J_i$ or two different $b\in J_i$ which produce equivalent twists.  This completes the proof of the theorem.
\end{proof}

\begin{theorem}\label{galois_fix}
Let $\mathcal{O}_K$ be the ring of integers of a real quadratic field $K$.   Let $x\in \mathcal{O}_K$ be such that $N(x)^2\leq \Delta_K/3$ and define $\mathcal{J} = (x)$.  If $\sigma(\mathcal{J}) = \mathcal{J}$ then $x$ extends to at most one good basis of $\mathcal{O}_K$, up to equivalence.
\end{theorem}
\begin{proof}
Express $x = a + c\omega$ in the canonical basis of $\mathcal{O}_K$.  If $a = 0$ then $x = \pm \omega$, and without loss of generality $x = \omega$ since we are only counting up to equivalence.  The inequality $N(x)^2\leq \Delta_K/3$ then forces $D = 5$ and hence $x = \omega = (1+\sqrt{5})/2$ is a unit in $\mathcal{O}_K$.  Replacing $x$ with $\omega^{-1}x = 1$ to obtain an equivalent basis, we reduce to the case $a\neq 0$.

As $\sigma(\mathcal{J}) = \mathcal{J}$, we must have $x = u\cdot \sigma(x)$ for $u = x/\sigma(x) \in \mathcal{O}_K^\times$.  Now let $y\in \mathcal{O}_K$ be any element such that $\{x,y\}$ is a good basis of $\mathcal{O}_K$.  We have the following sequence of equivalences:
\begin{equation}\label{galois_same}
\{x,y\}\sim \{u\cdot \sigma(x),u\cdot \sigma(y)\} = \{x,u\cdot\sigma(y)\}\sim \{x,-u\cdot\sigma(y)\}.
\end{equation}
If we let $y' = -u\cdot\sigma(y)$ then one checks that
\begin{equation}
y' = b' + d'\omega\quad\text{where}\quad b' = -\frac{aTr(\omega) + 2cN(\omega)}{N(x)} - b.
\end{equation}

Let $J_1$ and $J_2$ the intervals as in the $a\neq 0$ case of the proof of Theorem \ref{thm2}.  Consider the involution of $\R$ defined by
\begin{equation}
h(z) = -\frac{aTr(\omega) + 2cN(\omega)}{N(x)} - z.
\end{equation}
One checks that $h(J_1) = J_2$ hence the map $h$ is a bijective map from $J_1$ to $J_2$ which is its own inverse.  Clearly $h(b) = b'$ and hence $h(b') = b$.  Therefore $b\in J_1$ if and only if $b'\in J_2$, and similarly $b\in J_2$ if and only if $b'\in J_1$.  Thus $\{x,y\}$ and $\{x,y'\}$ are the two good bases as in the proof of statement (i).  But by (\ref{galois_same}), they are equivalent and we conclude that the ideal $\mathcal{J} = (x)$ only produces one good basis of $\mathcal{O}_K$, up to equivalence.
\end{proof}

The following corollary without the uniqueness statement appears in \cite[Examples following Corollary 3.2]{bayer_nebe}, wherein the authors construct these bases by hand.

\begin{corollary}\label{thm1}
There exists a unique good basis of $\mathcal{O}_K$ of the form $B = \{1,y\}$, up to equivalence.  The element $y$ is given by $y = b+\omega$ where $b = \lfloor \beta \rfloor$ and
\begin{equation}
\beta = \frac{1-Tr(\omega)-\sqrt{\Delta_K-3}}{2}.
\end{equation}
The minimal angle of the resulting well-rounded twist of $\Lambda_K$ is given by
\begin{equation}
\cos\theta_{T_\alpha B} = \left\{\begin{array}{cc}
\frac{b^2 + b + (5-D)/4}{2b + 1} & D\equiv 1\ (\text{mod } 4) \\
\frac{b^2+1-D}{2b} & D\not\equiv 1\ (\text{mod } 4)
\end{array}\right.
\end{equation}
\end{corollary}
\begin{proof}
This follows by applying Theorem \ref{galois_fix} to the element $x = 1\in\mathcal{O}_K$.  The calculation of $\cos \theta_{T_\alpha B}$ is immediate from Proposition \ref{cos}.
\end{proof}

\begin{corollary}\label{can_basis_not_good}
The canonical basis $B = \{1,\omega\}$ of $\mathcal{O}_K$ is good for twisting if and only if $K = \Q(\sqrt{5})$.  
\end{corollary}
\begin{proof}
If $K = \Q(\sqrt{5})$, then Example \ref{first_example} showed that the canonical basis of $\mathcal{O}_K$ is good for twisting.  For the other direction, suppose that the canonical basis $B$ is good for twisting.  If $D\not\equiv 1$ (mod $4$), then $Tr(\omega) = 0$ and $F_1(x,y) = F_2(x,y) = 1 - D = 0$, which is clearly impossible.  If $D\equiv 1$ (mod $4$), then $N(\omega) = \frac{1-D}{4}$ and $Tr(\omega) = 1$.  We have $F_1(1,\omega) = (3-D)/4$, and hence we must have $F_1(1,\omega)\leq0$. The two inequalities $F_1(1,\omega)\leq 0$ and $F_2(1,\omega)\geq 0$ are easily seen to be equivalent to $3\leq D\leq 7$, hence $D = 5$ as claimed.
\end{proof}

We finish this subsection by demonstrating a connection with the regulator.  Recall that for a lattice $\Lambda\subseteq\R^2$, the set $w(\Lambda) = \gamma(\Lambda)\cap\mathcal{W}_2$ denotes the set of all (similarity classes of) well-rounded twists of $\Lambda$.

\begin{theorem}\label{estimate}
Ordering real quadratic fields by their discriminant, we have
\begin{equation}
|w(\Lambda_K)| \leq \frac{4R_K}{\sqrt{3}} + O\left(\Delta_K^{1/4}\right)
\end{equation}
as $\Delta_K\rightarrow\infty$.
\end{theorem}
\begin{proof}
Let us define $\mathcal{P}_K$ to be the set of principal ideals contained in $\mathcal{O}_K$ which can possibly yield good bases:
\begin{equation}
\mathcal{P}_K = \left\{
(x)\subseteq \mathcal{O}_K: |N(x)|\leq \sqrt{\Delta_K/3},\ n\mathcal{O}_K\nmid(x) \text{ for $n\neq \pm1$}
\right\}.
\end{equation}
Since the process for calculating $w(\Lambda_K)$ outlined in Theorem \ref{thm2} assigns to every element of $\mathcal{P}_K$ at most two good bases, we have the obvious formula
\begin{equation}\label{obvious_bound}
|w(\Lambda_K)|\leq 2\cdot |\mathcal{P}_K|.
\end{equation}
Ignoring the divisibility condition in the definition of $\mathcal{P}_K$, we have by the Class Number Formula \cite[Section VI.3, Theorem 3]{lang_ant} that
\begin{equation}
|\mathcal{P}_K| \leq \frac{4R_K}{2\sqrt{\Delta_K}}\sqrt{\Delta_K/3} + O\left(\Delta_K^{1/4}\right) = \frac{2R_K}{\sqrt{3}} + O\left(\Delta_K^{1/4}\right)
\end{equation}
which when combined with (\ref{obvious_bound}) gives the result.
\end{proof}

\subsection{An Example}

\begin{example} We compute all of the well-rounded twists of $\Lambda_K$ for $K = \Q(\sqrt{201})$.  We begin by considering the principal ideal $\mathcal{J}_2 = (129-17\omega)$ which has $N(\mathcal{J}_2) = 2$.  We show explicitly how to extend $x = 129-17\omega$ to all possible good bases $\{x,y = b+d\omega\}$.  We have $d = (1+bc)/a = (1-17b)/129$ and therefore $y = b + \frac{1-17b}{129}\omega$.  The polynomials $f_i(b) = F_i(x,y)$ are given by
\begin{equation}
\begin{aligned}
f_1(b) &= \frac{1}{129^2}(-2b^2 + 1571b + 169277/2)\\
f_2(b) &= \frac{1}{129^2}(-2b^2 + 2087b - 302605/2)
\end{aligned}
\end{equation}
and thus the intervals $J_1$ and $J_2$ are given by
\begin{equation}
\begin{aligned}
J_1 &= \left[\frac{1571-387\sqrt{21}}{4},\frac{2087-387\sqrt{21}}{4}\right]\\
J_2 &= \left[\frac{1571+387\sqrt{21}}{4},\frac{2087+387\sqrt{21}}{4}\right]
\end{aligned}
\end{equation}
which both have width exactly $a = 129$.  There exists exactly one integer $b$ in each of $J_1$ and $J_2$ such that $d = (1+bc)/a\in \Z$, namely $b = 38$ and $b = 941$, respectively.  These values of $b$ give us the following exhaustive list of all good bases (up to equivalence) corresponding to $\mathcal{J}_2$:
\begin{equation}
B_1(\mathcal{J}_2) = \{129-17\omega, 38-5\omega\},\quad B_2(\mathcal{J}_2) = \{129-17\omega, 941-124\omega\}.
\end{equation}

We now follow this same algorithm for all principal ideals up to Galois conjugation with norm $\leq \lfloor \sqrt{\Delta_K/3}\rfloor = 8$, generated by elements $x = a + c\omega$ with $\gcd(a,c) =1 $.  We obtain all good bases of $\mathcal{O}_K$ up to equivalence, collected in Table \ref{q201table}.  In the left column, we list all principal ideals of norm $\leq 8$ up to Galois conjugation which are not divisible by any integer $>1$, in the center column the resulting good bases, and in the right column the cosine of the angle between the basis vectors of the resulting well-rounded twist.
\begin{table}
\centering
\begin{tabular}{|c|c|c|}
\hline
$\mathcal{J}_N = (x)$ & good bases $B_i(\mathcal{J}_N)$ & $|\cos\theta_i|$ \\
\hline\hline
\multirow{2}{*}{$\mathcal{J}_1 = (1)$} & $\{1,8-\omega\}$ & $7/15$ \\ \cline{2-3}
& $-$ & $-$ \\
\hline
\multirow{2}{*}{$\mathcal{J}_2 = (129-17\omega)$} & $\{129-17\omega, 38-5\omega\}$ & $1/3$ \\ \cline{2-3}
& $\{129-17\omega, 941-124\omega\}$ & $2/13$ \\
\hline
\multirow{2}{*}{$\mathcal{J}_3 = (941-124\omega)$} & $\sim B_1(\mathcal{J}_2)$ & $1/3$ \\ \cline{2-3}
& $-$ & $-$ \\
\hline
\multirow{2}{*}{$\mathcal{J}_4 = (38-5\omega)$} & $\{38-5\omega,15-2\omega\}$ & $1/11$ \\ \cline{2-3}
& $\sim B_2(\mathcal{J}_2)$ & $2/13$ \\
\hline
\multirow{2}{*}{$\mathcal{J}_5 = (13+2\omega)$} & $\{13+2\omega,33+5\omega\}$ & $1/9$ \\ \cline{2-3}
& $\sim B_1(\mathcal{J}_4)$ & $1/11$ \\ 
\hline
\multirow{2}{*}{$\mathcal{J}_6 = (8-\omega)$} & $\sim B_1(\mathcal{J}_1)$ & $7/15$ \\ \cline{2-3}
& $\sim B_1(\mathcal{J}_5)$ & $1/9$ \\ 
\hline
\multirow{1}{*}{$\mathcal{J}_8 = (6+\omega)$} & $-$ & $-$ \\
\hline
\end{tabular}
\caption{All good bases, up to equivalence, of the lattice $\Lambda_K$ for $K = \Q(\sqrt{201})$.  Here the ideal $\mathcal{J}_N$ has norm $N$.  For brevity we define $\theta_i = \theta_{T_\alpha B_i(\mathcal{J}_N)}$ to be the minimal angle of the resulting well-rounded twist.  }\label{q201table}
\end{table}

Since $\mathcal{J}_1$ and $\mathcal{J}_3$ are fixed by the Galois action (the ideal $\mathcal{J}_3$ being a prime above $3$, which ramifies), they only each produce one good basis by part (ii) of Theorem \ref{thm2}.  Several good bases appear twice in Table \ref{q201table}, since if $\{x,y\}$ is a good basis then clearly so is $\{y,x\}$, and thus running the above procedure for both of the principal ideals $(y)$ and $(x)$ will output the same good basis twice, up to equivalence.

The principal ideal $\mathcal{J}_8$ fails to produce any well-rounded bases because the intervals $J_1$ and $J_2$ constructed in the proof of Theorem \ref{thm2} overlap.  Thus the intervals on which the $f_i(b)$ have opposite signs have width less than $a$, and hence a $b\in \Z$ such that the $f_i(b)$ have opposite signs and $d = (1+bc)/a$ is an integer is not guaranteed to exist.
\end{example}

\section{An Infinite Family of $\Lambda_K$ with a Unique Well-Rounded Twist}\label{ortho_family}

In this section, we construct an infinite family of lattices of the form $\Lambda_K$ such that the orthogonal lattice is the \emph{only} well-rounded twist of $\Lambda_K$.  The next result can also be easily deduced from \cite[Examples following Corollary 3.2, and Theorems 3.2 and 3.3]{bayer_nebe}, though we present our own proof here to demonstrate the utility of our approach.  

\begin{lemma}\label{orthogonal_basis}
The well-rounded twist of $\Lambda_K$ corresponding to the good basis $\{1,b+\omega\}$ of Corollary \ref{thm1} is orthogonal if and only if either: (i) $D = s^2 + 1$ for some $s\in \Z$ and $D\equiv 2\pmod 4$, or (ii) $D = s^2 + 4$ for some $s\in \Z$ and $D\equiv 1\pmod 4$.
\end{lemma}
\begin{proof}
Suppose first that $D\not\equiv1\pmod4$, so that $Tr(\omega) = 0$ and $\Delta_K = 4D$, and $D\equiv 2\pmod{4}$ by Proposition \ref{orthog12}.  By Corollary \ref{thm1}, the resulting well-rounded twist corresponding to the good basis $\{1,b+\omega\}$ is orthogonal if and only if $D = b^2 + 1$, where
\begin{equation}
b = \lfloor \beta\rfloor\quad\text{and}\quad\beta=\frac{1-\sqrt{4D-3}}{2}.
\end{equation}
Equivalently, the twisted lattice is orthogonal if and only if $D-1$ is a square in $\Z$ and $\lfloor\beta\rfloor = -\sqrt{D-1}$.  One verifies easily that for any $D$, we have $|\beta - (-\sqrt{D-1})|<1$, hence if $D-1$ is a square it follows that $\lfloor \beta\rfloor = -\sqrt{D-1}$.  Thus for $D\not\equiv1\pmod 4$, the resulting well-rounded twist from the good basis $\{1,b+\omega\}$ is orthogonal if and only if $D = s^2 + 1$.

When $D\equiv1\pmod 4$, the proof is similar.  From Corollary \ref{thm1} we see that $\cos\theta_{T_\alpha B}=0$ if and only if $b = \lfloor\beta\rfloor = \frac{1}{2}(-1-\sqrt{D-4})$.  Again, one verifies that $|\beta - \frac{1}{2}(-1-\sqrt{D-4})|<1$ for any $D$, and that for any $D\equiv1\pmod 4$ such that $D-4$ is a square in $\Z$ we also have $\frac{1}{2}(-1-\sqrt{D-4})\in \Z$.  We conclude that the twisted lattice is orthogonal if and only if $D-4$ is a square in $\Z$ and $D = s^2 + 4$ for some $s\in \Z$.
\end{proof}

The following result and subsequent proof is a mild generalization of a result appearing in \cite{hasse_chowla} by Ankeny, Chowla, and Hasse, to whom the result was originally communicated by Davenport.

\begin{lemma}\label{no_small_norms}
Let $D$ be a positive square-free integer such that either (i) $D = s^2 + 1$ for some $s\in \Z$ and $D\equiv 2\pmod 4$, or (ii) $D = s^2+4$ for some $s\in\Z$ and $D \equiv 1\pmod 4$.  For $K = \Q(\sqrt{D})$ and $x\in \mathcal{O}_K$, either
\begin{equation}
|N(x)|>\sqrt{\Delta_K/3}\quad \text{or}\quad x = \varepsilon n
\end{equation}
for $\varepsilon\in \mathcal{O}_K^\times$ and some $n\in \Z$.
\end{lemma}
\begin{proof}
The result for $D\equiv 2\pmod{4}$ is \cite[Lemma]{hasse_chowla}, thus we may assume $D\equiv 1\pmod{4}$.  We will use a similar argument to prove the result for $D = s^2 + 4$, when $D\equiv 1 \pmod 4$.  Firstly, we note that $s=2t+1$ must be odd, else $D$ is divisible by $4$.  Note that we can take $t\geq0$.  Let us take $\varepsilon = t + \omega$, and choose $x = a - c\bar\omega$ not associate to any integer, such that among all associates of $x$, $c$ is positive and minimal.  Computing the associate $\varepsilon x$ of $x$ yields
\begin{equation}
\varepsilon x = at + cN(\omega)  + (a-c(t+1))\omega
\end{equation}
and replacing $\varepsilon$ with $-\varepsilon$ if necessary, we are free to assume $a-c(t+1)>0$.  By the minimality of $c$ we therefore have $0<c\leq a-c(t+1)$, and hence $0<c(t+2)\leq a$, so that in particular $a>0$ and $a^2\geq c^2(t+2)^2$.  We now use these inequalities to bound $N(x)$ below by
\begin{equation}
N(x) = a^2 + ac + c^2N(\omega) > c^2(t+2)^2 + c^2N(\omega) = 3c^2(t+1)\geq 3(t+1)
\end{equation}
Squaring both sides yields
\begin{equation}
N(x)^2 > 9(t+1)^2 > (4t^2 + 4t + 5)/3 = \Delta_K/3
\end{equation}
which is what we wanted to prove.
\end{proof}

We recall an elementary lower bound on the regulator of a real quadratic field.  We refer to \cite[Section 1]{regulator_bound} for a proof of the case $D\equiv 1\pmod{4}$, and the case $D\not\equiv1\pmod{4}$ is proved in an almost identical manner.  
\begin{lemma}\label{lemma_regulator_bound}
If $R_K$ is the regulator of $K = \Q(\sqrt{D})$, then we have
\begin{equation}\label{regulator_bound}
R_K \geq \left\{\begin{array}{ll}
\log\left|\frac{1}{2}\left(\sqrt{D-4}+\sqrt{D}\right)\right| &\text{if $D\equiv 1\pmod 4$} \\
\log|\sqrt{D-1}+\sqrt{D}| & \text{if $D\not\equiv1\pmod4$}
\end{array}\right.
\end{equation}
and we have equality if and only if $D$ is of the form $D = s^2 +4$ in the case $D\equiv 1\pmod{4}$, or $D = s^2 + 1$ in the case $D\not\equiv1\pmod{4}$.
\end{lemma}

We can now state and prove the main theorem of this section.  

\begin{theorem}\label{big_theorem}
Let $K$ be a real quadratic field and $\Lambda_K$ the lattice given by the canonical embedding of its ring of integers.  Then the following are equivalent:
\begin{itemize}
\item[(i)] The only well-rounded twist of $\Lambda_K$ is the orthogonal lattice.
\item[(ii)] $D=s^2+1$ for some $s\in \Z$ and $D\equiv 2$ (mod $4$), or $D=s^2 + 4$ for some $s\in \Z$ and $D\equiv 1$ (mod $4$).
\item[(iii)] The regulator $R_K$ of $K$ meets the lower bound (\ref{regulator_bound}) with equality.
\end{itemize}
\end{theorem}
\begin{proof}
(i) $\Rightarrow$ (ii) Suppose that the only well-rounded twist of $\Lambda_K$ is the orthogonal lattice.  Since Corollary \ref{thm1} shows us that we always have a good basis of the form $\{1,y\}$, this basis must produce an orthogonal twist.  By Lemma \ref{orthogonal_basis}, $D$ is then of the stated form.

(ii) $\Rightarrow$ (i) Suppose that $D = s^2 + 1$ and $D\equiv 2\pmod 4$, or that $D = s^2 + 4$ and $D\equiv 1\pmod 4$.  Let $\{x,y\}$ be a good basis of $\Lambda_K$.  Then $|N(x)|\leq \sqrt{\Delta_K/3}$ by Proposition \ref{norm_cond}, which by Lemma \ref{no_small_norms} implies that $x = u$ for some unit $u\in \mathcal{O}_K^\times$.  Hence up to equivalence, the only good basis is the basis $\{1,y\}$ constructed in the proof of Theorem \ref{thm1}.  As we saw in Lemma \ref{orthogonal_basis}, the resulting well-rounded twist of the good basis $\{1,y\}$ is orthogonal.

(ii) $\Leftrightarrow$ (iii) This is contained in the statement of Lemma \ref{lemma_regulator_bound}.
\end{proof}

For example, in the $D\equiv 1\pmod 4$ case, setting $s = 1$ proves that the only well-rounded twist of the ring of integers of $\Q(\sqrt{5})$ is the orthogonal lattice, as suggested by Fig.\ \ref{TL_orbit}.


\begin{corollary}
There exists infinitely many real quadratic fields $K = \Q(\sqrt{D})$ for both $D\equiv 1\pmod 4$ and $D\equiv 2\pmod 4$, such that the only well-rounded twist of $\Lambda_K$ is the orthogonal lattice.
\end{corollary}

\begin{proof}
According to Theorem \ref{big_theorem} (ii), we need only to show that there exists infinitely many square-free $D$ such that $D = s^2 + 1$ and $D\equiv 2\pmod 4$, or $D = s^2 + 4$ and $D\equiv 1\pmod 4$.  It is known that infinitely many such $D$ exist, and even form a subset of $\Z$ of positive density; see \cite{ricci} and also \cite[Introduction and Theorem 1]{heath_brown} for a modern summary of related results.
\end{proof}

We conclude this subsection by comparing our result with a result of Bayer-Fluckinger and Nebe \cite{bayer_nebe}:
\begin{theorem}[\cite{bayer_nebe}, Theorem 3.2] 
The orthogonal lattice is a twist of $\Lambda_K$ if and only if the fundamental unit of $\mathcal{O}_K$ has norm $-1$.
\end{theorem}
The difference between this result and Theorem \ref{big_theorem} can be illustrated by considering the field $K = \Q(\sqrt{17})$, whose geodesic $\gamma(\Lambda_K)$ is depicted in Fig.\ \ref{TL_orbit}.  This field has fundamental unit $\varepsilon = 4 + \sqrt{17}$ of norm $N(\varepsilon) = -1$, thus $\Lambda_K$ has an orthogonal twist.  On the other hand, $K$ clearly fails condition (ii) of Theorem \ref{big_theorem}, hence the orthogonal lattice is not the \emph{only} well-rounded twist of $\Lambda_K$.  And indeed, using our Theorem \ref{thm2} we obtain two good bases of $\mathcal{O}_K$, namely $B_1 = \{1,1+\omega\}$, which gives a well-rounded twist with $|\cos\theta| = 1/3$, and $B_2 = \{1+\omega,2+\omega\}$, which gives the orthogonal twist.


\section{Ideal Lattices from Markoff Numbers}\label{markoff_lattices}

\subsection{The Minimum Product Distance of an Ideal Lattice}

We now give a purely number-theoretic description of the minimum product distance of an ideal lattice.  

\begin{prop}\label{mpd}
Let $\Lambda_\mathcal{I}$ be an ideal lattice where $\mathcal{I}\subseteq\mathcal{O}_K$ for $K$ a totally real number field, scaled so that $\vol(\Lambda_\mathcal{I}) = 1$.  Then
\begin{align}
N(\Lambda_\mathcal{I}) = \min_{\substack{\mathcal{J}\subseteq \mathcal{O}_K \\ [\mathcal{J}] = [\mathcal{I}]^{-1}}} \frac{N(\mathcal{J})}{\sqrt{\Delta_K}}.
\end{align}
In particular, $N(\Lambda_\mathcal{I})$ depends only on the ideal class $[\mathcal{I}]\in C_K$, and $N(\Lambda_\mathcal{I}) = \frac{1}{\sqrt{\Delta_K}}$ for any principal ideal $\mathcal{I}\subseteq\mathcal{O}_K$.
\end{prop}
\begin{proof}
Recall that $(x)\subseteq \mathcal{I}$ if and only if there exists another ideal $\mathcal{J}$ such that $\mathcal{I}\cdot \mathcal{J} = (x)$.  We therefore get a bijection between norms of principal ideals contained in $\mathcal{I}$ and the set of all integers of the form  $N(\mathcal{I})N(\mathcal{J})$ where $[\mathcal{J}] = [\mathcal{I}]^{-1}$ in the class group $C_K$.  We compute that
\begin{align}
N(\Lambda_\mathcal{I}) = \frac{1}{N(\mathcal{I})\sqrt{\Delta_K}}\min_{0\neq x\in \Lambda_\mathcal{I}} |N(x)| = \min_{\substack{\mathcal{J}\subseteq \mathcal{O}_K \\ [\mathcal{J}] = [\mathcal{I}]^{-1}}} \frac{N(\mathcal{J})}{\sqrt{\Delta_K}}
\end{align}
which is what was claimed.
\end{proof}

The well-rounded twists of the lattices $\Lambda_K$ considered in Theorem \ref{big_theorem} thus produce an infinite family of well-rounded lattices $\Lambda$ such that 
\begin{itemize}
\item[(i)] $\cos\theta = 0$, and 
\item[(ii)] $N(\Lambda) = N(\Lambda_K)= \frac{1}{\sqrt{\Delta_K}}\rightarrow 0$ as $\Delta_K\rightarrow \infty$.  
\end{itemize}
Thus while interesting from a number-theoretic perspective, this family of ideal lattices is, in some sense, the worst possible family to consider for the purposes of making both the sphere-packing radius and minimum product distance large.

More generally, if one only considers ideal lattices arising from principal ideals as is usually done in e.g.\ \cite{oggier_belfiore}, one is only ever able to construct infinite families of lattices $\Lambda$ such that $N(\Lambda)\rightarrow0$.  Thus the consideration of $\Lambda_\mathcal{I}$ for non-principal $\mathcal{I}$ is necessary to construct infinite families of lattices with $N(\Lambda)$ bounded below by a positive constant.

In \cite{srinivasan}, A.\ Srinivasan proves the following theorem, which improves on the classical Minkowski bound for real quadratic fields.

\begin{theorem}[\cite{srinivasan}, Theorems 1.1 and 1.2]\label{improve_mink}
Let $K$ be a real quadratic field with discriminant $\Delta_K$ and class group $C_K$.  Then for all ideal classes $[\mathcal{J}]\in C_K$, there exists a representative $\mathcal{I}\in [\mathcal{J}]$ such that 
\begin{equation}
N(\mathcal{I})\leq S_K:=1 + \left\lfloor \frac{\sqrt{\Delta_K}}{3}\right\rfloor.
\end{equation}
Conversely, there exist infinitely many real quadratic fields $K$ which contain ideal classes $[\mathcal{J}]$ whose least norm representative satisfies $N(\mathcal{I}) = S_K$.
\end{theorem}

As an immediate corollary of the above two results, we obtain:
\begin{corollary}
Let $K$ be a real quadratic field, let $\mathcal{I}\subseteq\mathcal{O}_K$ be an ideal in the ring of integers, and consider the ideal lattice $\Lambda_\mathcal{I}$.  Then
\begin{equation}\label{norm_bound}
N(\Lambda_\mathcal{I}) \leq \widehat{S}_K:=\frac{1}{\sqrt{\Delta_K}}\left(
1 + \left\lfloor
\frac{\sqrt{\Delta_K}}{3}
\right\rfloor
\right).
\end{equation}
\end{corollary}

The infinite family of ideals meeting the bound $S_K$ is of particular interest to us, since they produce an infinite family of ideal lattices which achieve the upper bound $\widehat{S}_K$.  Note that for $\Delta_K>5$ the bound $S_K$ is larger than $1$, so in particular all ideal classes meeting the bound $S_K$ are non-trivial, except the trivial class in $\Q(\sqrt{5})$.

Before delving into constructions of ideal lattices from Markoff numbers, we digress momentarily to prove the following theorem.  This result appears to at least be implicitly believed in the literature (see \cite{oggier_belfiore,oggier_belfiore_bayer}), but the authors know of no previous proof.  As the authors of \cite{oggier_belfiore_bayer} loosely conjectured, considering non-principal ideals does not improve the minimum product distance, at least in dimension $n = 2$.  
\begin{theorem}\label{best_mpd}
Among all ideal lattices $\Lambda_\mathcal{I}$ where $\mathcal{I}$ is an ideal of the ring of integers of a totally real quadratic number field, the one with maximal $N(\Lambda_\mathcal{I})$ is $\mathcal{I} = \mathcal{O}_K$ where $K = \Q(\sqrt{5})$.
\end{theorem}
\begin{proof}
We have $N(\Lambda_{\mathcal{O}_{\Q(\sqrt{5})}}) = 1/\sqrt{5}$.  We have $\widehat{S}_K\geq1/\sqrt{5}$ for only finitely many values of $\Delta_K$, and hence only finitely many fields $K$.  One checks easily that all such $K$ are of the form $K = \Q(\sqrt{D})$ with $D<100$.  As the value of $N(\Lambda_\mathcal{I})$ depends only on the ideal class $[\mathcal{I}]\in C_K$, one only needs to check that $N(\Lambda_{\mathcal{I}})<1/\sqrt{5}$ for finitely many ideals.  This is easily done using Sage \cite{sagemath}.
\end{proof}

As is noted in \cite{oggier_belfiore}, this theorem is obvious if one restricts to principal ideal lattices, since in that case the question of maximizing $N(\Lambda_K)$ is equivalent to finding the real quadratic field with smallest discriminant.  Note the necessity of the improved bound $S_K$ in the above proof: applying the same argument with the classical Minkowski bound $M_K = \frac{\sqrt{\Delta_K}}{2}$, we would have $\frac{1}{\sqrt{\Delta_K}}M_K = \frac{1}{2}>\frac{1}{\sqrt{5}}$ for all $K$, and thus one cannot reduce the problem to checking finitely many ideal classes.

In broad terms, Theorem \ref{best_mpd} states that if one restricts their attention to well-rounded twists of ideal lattices, the lattice that maximizes $N$, namely $\Lambda_{\Q(\sqrt{5})}$, is not that which maximizes $\rho$.  Indeed, the hexagonal lattice is obtained as a well-rounded twist of $\Lambda_{\Q(\sqrt{3})}$.  In the next section we investigate the trade-off between these two quantities by studying the sphere-packing radii of four infinite families of lattices which, in a sense, have the maximal limiting value of $N$.

Lastly, we remark that we have not ruled out the possibility that some ideal lattice coming from an ideal in a non-maximal order has larger minimum product distance; however, this seems unlikely as increasing the conductor of an order increases the discriminant without any obvious benefit.

\subsection{Markoff Numbers and Markoff Ideals}

In \cite{srinivasan}, the author shows that ideal classes which attain the bound $S_K$ come from Markoff numbers.  As a general reference for Markoff numbers we refer to \cite[Chapter II]{cassels_diophantine}, which contains all of the fundamental definitions and results we will require.  So that our paper is self-contained, we review the basics below.

Given three non-negative integers $a$, $b$, and $c$, we call $(a,b,c)$ a \emph{Markoff triple} if the equation
\begin{equation}\label{market}
a^2 + b^2 + c^2 = 3abc
\end{equation}
is satisfied, and any $a$, $b$, or $c$ appearing in such a triple is a \emph{Markoff number}. The first few Markoff numbers are $c = 1, 2, 5, 13, 29, \ldots$ which appear in the solutions $(1,1,1)$, $(1,1,2)$, $(1,2,5)$, $(1,5,13)$, and $(2,5,29)$ to the Markoff equation (\ref{market}).  The \emph{Markoff Conjecture} asserts that given a solution to the above equation such that $a\leq b\leq c$, $c$ completely determines $a$ and $b$.  

It is well-known that there are infinitely many odd Markoff numbers, and any odd Markoff number $c$ must satisfy $c\equiv 1\pmod{4}$.  For simplicity, we will restrict to the case of odd $c$.  The following construction generalizes to even Markoff numbers, but the resulting ideal belongs to a quadratic order which is not necessarily maximal.  This construction follows \cite[Chapter II.4]{cassels_diophantine}.

Let $(a,b,c)$ be a Markoff triple with $c$ odd.  Let $D = 9c^2 - 4$ and let $K = \Q(\sqrt{D})$.  Clearly $D\equiv 1\pmod{4}$.  One constructs an ideal class whose minimal norm representative $\mathcal{I}\subseteq\mathcal{O}_K$ satisfies $N(\mathcal{I}) = S_K$ as follows.  Define integers $k$ and $\ell$ by
\begin{equation}
ak\equiv b\pmod{c},\ \ 0\leq k<c,\quad \text{and}\quad k^2+1 = \ell c
\end{equation}
Then the quadratic form
\begin{equation}
Q(X,Y) = cX^2 + (3c-2k)XY + (\ell-3k)Y^2
\end{equation}
has discriminant $\Delta_{Q} = 9c^2-4=D$, and so $Q$ defines an element of the class group $C_K$.  The minimal norm representative $\mathcal{I}_c\subseteq\mathcal{O}_K$ of this ideal class is given in its canonical basis as
\begin{equation}\label{markoff_basis}
\mathcal{I}_c = \left\{\begin{array}{cc}
(c,k-(c+1)/2+\omega) & \text{if } k\geq (c+1)/2 \\
(c,k+(c-1)/2+\omega) & \text{if } k<(c+1)/2
\end{array}
\right.
\end{equation}
One verifies that this is indeed the canonical basis of this ideal by checking the conditions of (\ref{ideal_can_basis}).  For example, if $c = 1$ then $D = 5$ and $k = 0$, and therefore we have $\mathcal{I}_1 = (1,\omega) = \mathcal{O}_K$ in the field $K = \Q(\sqrt{5})$.

For an odd Markoff number $c$, we will call the ideal $\mathcal{I}_c$ a \emph{Markoff ideal}.  Markoff ideals provide us with an infinite family of non-principal ideals with good minimum product distance.  More specifically, if one is willing to assume the Markoff Conjecture, then the ideal classes $[\mathcal{I}_c]$ all achieve the bound $S_K$ of Theorem \ref{improve_mink}.

\begin{theorem}[See \cite{srinivasan}, Theorem 3.2]\label{markoff_mpd}
Let $c$ be an odd Markoff number, let $K = \Q(\sqrt{D})$ with $D = 9c^2-4$, and let $\mathcal{I}_c$ be the corresponding Markoff ideal in $\mathcal{O}_K$.  Let $\mathcal{J}_c$ be any representative of the inverse class $[\mathcal{I}_c]^{-1}\in C_K$ and consider the ideal lattice $\Lambda_c = \Lambda_{\mathcal{J}_c}$.  Then $N(\Lambda_c) = \widehat{S}_K >1/3$ and $\lim_{c\rightarrow\infty}N(\Lambda_c) = 1/3$.

Furthermore, if the Markoff Conjecture is true, then $[\mathcal{I}_c]^{-1} = [\mathcal{I}_c]$ in $C_K$ and thus one can simply set $\Lambda_c = \Lambda_{\mathcal{I}_c}$.
\end{theorem}
\begin{proof}
By the main results of \cite{srinivasan} we have $N(\mathcal{I}_c) = S_K$ and therefore $N(\Lambda_c) = \widehat{S}_K > 1/3$ as claimed.  The statement about the limit is easy to show.
\end{proof}

In what follows, we explicitly assume the Markoff Conjecture.  We now define $\Lambda_c = \psi(\mathcal{I}_c)$ for any odd Markoff number $c$, and refer to this as a \emph{Markoff lattice}.  The family of Markoff lattices is therefore an infinite family of lattices such that $N(\Lambda_c) > 1/3$ and $\lim_{c\rightarrow\infty}N(\Lambda_c) = 1/3$.  Furthermore the results of \cite{srinivasan} essentially show that $1/3$ is the largest constant one could hope for from a family of ideal lattices.

\subsection{Well-Rounded Twists of Markoff Lattices}



While the Markoff lattices $\Lambda_c$ are an infinite family of lattices with the largest possible minimum product distance among ideal lattices, we must consider well-rounded twists of Markoff lattices to also guarantee good sphere-packing properties.  In what follows we assume that $c>1$, the case $c = 1$ being already well-understood since $\mathcal{I}_1 = \mathcal{O}_K$ where $K = \Q(\sqrt{5})$, and we have already seen that $\Lambda_1$ has exactly one well-rounded (orthogonal) twist.  

When $c > 1$, the algorithm of Theorem \ref{thm2} using the element $x = c$ and the basis $\{u,v\} = \{c,k + (c-1)/2+\omega\}$ extends the element $c$ to two good bases $B_i$ of the form
\begin{equation}\label{good_Markoff_bases}
B_i = \{c,\beta_ic+k+(c-1)/2+\omega\},\quad\text{where}\quad \beta_i = \left\lfloor\frac{-k\pm\sqrt{3c^2/2-1}}{c}\right\rfloor
\end{equation}
and we choose the negative sign for $\beta_1$.  Setting $\theta_i$ to be the minimal angle of the resulting well-rounded twist, we have
\begin{equation}\label{cos_good_Markoff_bases}
\cos\theta_i = \frac{(\beta_i^2 + \beta_i -1)c^2 + (2\beta_i+1)kc + k^2 + 1}{(2\beta_i+1)c^2 + 2kc}
\end{equation}
For some special Markoff triples, we can say more.  We define the Fibonacci numbers $F_n$ and the Pell numbers $P_n$ by the recurrence relations
\begin{equation}
F_1 = F_2 = 1,\ F_n = F_{n-1}+F_{n-2},\quad P_1 = 0,\ P_2 = 1,\ P_n = 2P_{n-1}+P_{n-2}.
\end{equation}

\begin{theorem}\label{fibonacci_lattices}
Suppose that $(1,b,c)$ is a Markoff triple with $c$ odd.  Then the good bases $B_i$ of (\ref{good_Markoff_bases}) which extend $c\in\mathcal{I}_c$ are given by
\begin{equation}
B_1 = \{c,k+(c-1)/2+\omega\},\quad B_2 = \{c,-2c+k+(c-1)/2+\omega\}
\end{equation}
and the corresponding minimal angles $\theta_i$ of the well-rounded twists satisfy
\begin{equation}
\cos\theta_1 = 0,\quad \lim_{c\rightarrow\infty}\cos\theta_2 = \frac{6-4\sqrt{5}}{11} = -0.2677\ldots
\end{equation}
Furthermore, there are infinitely many Markoff triples of the form $(1,b,c)$ where $c$ is odd.  Such triples can be given as $(1,F_{2n-1},F_{2n+1})$ where $F_{2n+1}$ is an odd and odd-indexed Fibonacci number.
\end{theorem}
\begin{proof}
Since $a = 1$, we have $k = b$.  Solving the Markoff equation $1 + b^2 + c^2 = 3bc$ for $b$ yields $b = (3c-\sqrt{5c^2-4})/2$, where we discard the other solution because $b\leq c$.  Computing the $\beta_i$ as in equation (\ref{good_Markoff_bases}) gives $\beta_1 = -2$ and $\beta_2 = 0$.  This shows that the given bases are of the stated form.

We compute $\cos\theta_i$ by plugging into (\ref{cos_good_Markoff_bases}).  For $\beta_1 = -2$ we obtain
\begin{equation}
\cos\theta_1 = \frac{1+b^2 + c^2 -3bc}{-3c^2 + 2bc} = 0
\end{equation}
as claimed.  For $\beta_2 = 0$ we obtain
\begin{equation}
\cos\theta_2 = \frac{4c^2 - 2c\sqrt{5c^2-4}}{4c^2-c\sqrt{5c^2-4}}\rightarrow\frac{4-2\sqrt{5}}{4-\sqrt{5}} = \frac{6-4\sqrt{5}}{11}
\end{equation}
as claimed.  Lastly, it is straightforward to verify that $(1,F_{2n-1},F_{2n+1})$ is a Markoff triple for any $n$, and that there are infinitely many odd and odd-indexed Fibonacci numbers.
\end{proof}

\begin{theorem}\label{pell_lattices}
Suppose that $(2,b,c)$ is a Markoff triple with $c$ odd.  Then the good bases $B_i$ of (\ref{good_Markoff_bases}) which extend $c\in\mathcal{I}_c$ are given by
\begin{equation}
B_1 = \{c,k+(c-1)/2+\omega\},\quad B_2 = \{c,-2c+k+(c-1)/2+\omega\}
\end{equation}
and the corresponding minimal angles $\theta_i$ of the well-rounded twists satisfy
\begin{equation}
\lim_{c\rightarrow\infty}\cos\theta_1 = \frac{3-\sqrt{2}}{7}=0.2265\ldots,\quad \lim_{c\rightarrow\infty}\cos\theta_2 = \frac{15-11\sqrt{2}}{17} = -0.0327\ldots
\end{equation}
Furthermore, there are infinitely many Markoff triples of the form $(2,b,c)$ where $c$ is odd.  Such triples can be given as $(1,P_{2n-2},P_{2n})$ where $P_{2n}$ is an even-indexed Pell number.
\end{theorem}

\begin{proof}
The proof is as in the previous theorem.  We solve for $b$ using the Markoff equation $4 + b^2 + c^2 = 6bc$ to obtain $b = 3c - 2\sqrt{2c^2-1}$, discarding the other solution because $b\leq c$.  To solve for $k$, we note that the equation $ak\equiv b\pmod{c}$ becomes
\begin{equation}
2k\equiv 3c - 2\sqrt{2c^2-1} \equiv -2\sqrt{2c^2-1}\pmod{c}
\end{equation}
and hence $k\equiv -\sqrt{2c^2-1}\pmod{c}$.  The condition $0\leq k<c$ then forces $k = 2c - \sqrt{2c^2-1}$.  Using (\ref{good_Markoff_bases}) to compute the $\beta_i$ yields $\beta_1 = -2$ and $\beta_2 = 0$, and using (\ref{cos_good_Markoff_bases}) to compute the corresponding values of $\cos\theta_i$ yields
\begin{align}
\cos\theta_1 &= \frac{c^2 - c\sqrt{2c^2-1}}{c^2-2c\sqrt{2c^2-1}}\rightarrow\frac{3-\sqrt{2}}{7} \\
\cos\theta_2 &= \frac{7c^2 - 5\sqrt{2c^2-1}}{5c^2-2\sqrt{2c^2-1}}\rightarrow\frac{15-11\sqrt{2}}{17}
\end{align}
as claimed.  Lastly, it is not hard to verify that $(2,P_{2n-2},P_{2n})$ is a Markoff number for every $n$, and that all even-indexed Pell numbers are odd.
\end{proof}

Theorems \ref{fibonacci_lattices} and \ref{pell_lattices} thus each construct two infinite families $\{\Lambda_{c,1}\}$ and $\{\Lambda_{c,2}\}$, namely the well-rounded twists of $\Lambda_c$ for the appropriate Markoff number $c$, which satisfy (again, assuming the Markoff Conjecture)
\begin{equation}
\lim_{c\rightarrow\infty}N(\Lambda_{c,i}) = \frac{1}{3},\quad \lim_{c\rightarrow\infty}\cos\theta_i = \left\{
\begin{array}{cl}
0 & i = 1,\ c \in \mathcal{F}_{\text{odd}}\\ 
\frac{6-4\sqrt{5}}{11} & i = 2,\ c\in\mathcal{F}_{\text{odd}}\\
& \\
\frac{3-\sqrt{2}}{7} & i = 1,\ c\in \mathcal{P}_{\text{even}} \\
\frac{15-11\sqrt{2}}{17} & i = 2,\ c\in\mathcal{P}_{\text{even}}
\end{array}\right.
\end{equation}
where $\theta_i$ is the minimal angle of $\Lambda_{c,i}$, $\mathcal{F}_{\text{odd}}$ is the set of odd and odd-indexed Fibonacci numbers, and $\mathcal{P}_{\text{even}}$ is the set of even-indexed Pell numbers.  Among these four families, the largest limiting value of $|\cos\theta|$ is $|(6-4\sqrt{5})/11|$, hence this family offers the best sphere-packing.

To conclude, we mention that of course it is possible to use the algorithm outlined in Theorem \ref{thm2} to compute \emph{all} well-rounded twists of the lattice $\Lambda_c$, to potentially improve the sphere-packing radius even further.  However, an extensive computer search using this algorithm did not reveal any well-rounded twists of any $\Lambda_c$ such that $|\cos\theta|> |(6-4\sqrt{5})/11|$, though we cannot prove that such a well-rounded twist of a Markoff ideal lattice does not exist.  Hence we content ourselves with the above explicit constructions.

\section*{Acknowledgements}

The authors would like to thank Piermarco Milione, Hunter Brooks, Lenny Fukshansky, Camilla Hollanti, and Guillermo Mantilla-Soler for helpful discussions regarding the results of this article.  We would also like to thank the anonymous reviewer, whose comments greatly improved the readability and presentation of this article.

\section*{References}
\bibliographystyle{elsarticle-num.bst}
\bibliography{references.bib}

\end{document}